\documentclass[12pt, a4paper]{article}
\usepackage[utf8]{inputenc}
\usepackage[english]{babel}
\usepackage{amsmath,verbatim,amsthm,mathrsfs,stmaryrd}
\usepackage{indentfirst}
\usepackage{amstext}
\usepackage{enumerate}
\usepackage{amsfonts}
\usepackage{textcomp}
\usepackage{amssymb}
\usepackage[dvips]{graphicx}
\usepackage{setspace}
\usepackage[dvipsnames]{xcolor}
\usepackage{graphicx}
\graphicspath{ {./images/} }
\usepackage[all]{xy}

\newcommand {\N}{\mathbb{N}}
\newcommand{\DD}{\mathsf{D}}

\newcommand {\Z}{\mathbb{Z}}
\newcommand{\homeo}{\operatorname{Homeo}}
\newcommand{\skel}[1]{^{(#1)}}

\newcommand {\GG}{\mathcal{G}}
\newcommand {\Gs}{G_{\sigma}}               

\newcommand {\G}{\mathcal{G}}

\newcommand {\La}{\Lambda}
\newcommand {\la}{\lambda}
\newcommand {\al}{\alpha}
\newcommand {\ap}{\alpha^{\prime}}
\newcommand {\be}{\beta}
\newcommand {\bp}{\beta^{\prime}}

\newcommand\supp{\mathrm{supp}}

\usepackage{geometry}
\DeclareMathOperator{\id}{id}

\newtheorem{teorema}{Theorem}[section]
\newtheorem{lema}[teorema]{Lemma}

\newtheorem{definicao}[teorema]{Definition}
\newtheorem{proposicao}[teorema]{Proposition}
\newtheorem{exemplo}[teorema]{Example}
\newtheorem{remark}[teorema]{Remark}

\begin{document}

\author{Gilles Gon\c{c}alves de Castro\footnote{Partially supported by Capes-PrInt Brazil grant number 88881.310538/2018-01.}, Daniel Gon\c{c}alves\footnote{Partially supported by Conselho Nacional de Desenvolvimento Cient\'ifico e Tecnol\'ogico (CNPq) grant numbers 304487/2017-1 and 406122/2018-0  and Capes-PrInt grant number 88881.310538/2018-01 - Brazil.}, and Daniel W van Wyk\footnote{This study was financed in part by the Coordenação de Aperfeiçoamento de Pessoal de Nível Superior - Brasil (CAPES) - Finance Code 001.}}

\title{Topological full groups of ultragraph groupoids as an isomorphism invariant}

\maketitle

\begin{abstract}
We prove two isomorphism-invariance theorems for groupoids associated with ultragraphs. These theorems
characterize ultragraphs for which the topological full group of an associated groupoid is an isomorphism invariant. These results extend those of graph groupoids to ultragraph groupoids while providing another concrete example where the topological full group of a groupoid is a complete isomorphism invariant.

\end{abstract}

\vspace{0.5pc}




\section{Introduction}

Ultragraphs are versatile combinatorial objects that encompass graphs. Introduced by Mark Tomforde in \cite{MR2050134} as an object to unify the study of graph and Exel-Laca algebras, ultragraphs have connections with branching systems \cite{MR3554458}, infinite alphabet shift spaces \cite{MR3600124, MR3938320, MR3918205}, chaos \cite{gonccalves2018li,gonccalves2019ultragraph}, Leavitt path algebras \cite{imanfar2017leavitt}, KMS states \cite{MR3856223}, AF algebras \cite{MR2541281}, groupoids \cite{MR2457327}, topological quivers \cite{MR2413313}, and are also interesting objects to study on their own. In this paper we  will focus on the connections with topological dynamics and groupoids. More precisely, we will use the recent description of topological full groups of ample groupoids with locally compact unit spaces given in \cite{MR3950815} to describe isomorphism of ultragraph groupoids (under Condition~(RFUM)) in terms of isomorphism of their respective topological full groups. 

The use of topological full groups as invariants for a certain form of equivalence between orbits of dynamical systems ranges from its application in Cantor minimal systems \cite{MR1710743}, to ample groupoids \cite{MR2876963, MR3950815}, passing through Cuntz-Krieger algebras \cite{MR3668049}  and graph algebras \cite{MR3950815} (to name a few). In the groupoid setting, recent results connect continuous orbit equivalence, diagonal preserving isormorphism and groupoid isomorphism (see \cite{carlsen2017reconstruction} for example) and in \cite{MR3950815} the authors add topological full groups (of groupoids with locally compact unit space) to the list. 

Beyond the general study of full groups, it is important to study them in specific cases. In fact, in \cite{MR3950815} the authors mention in the introduction that their initial goal was to study the topological full groups of general graph groupoids and add them to the list of invariants for continuous orbit equivalence between graphs. Our paper moves in this direction, as we characterize the full groups associated to ultragraph groupoids and use the general results in \cite{MR3950815} to describe the topological full groups of ultragraph groupoids as invariants for isormorphism of such groupoids. As with the graph case, due to results in \cite{carlsen2017reconstruction, phdthesisFelipe, tascagoncalves}, our results connect topological full groups  with continuous orbit equivalence of ultragraph shift spaces and diagonal preserving isomorphism between ultragraph C*-algebras. 

The paper is organized in the following way. Section \ref{sec1} contains basic definitions and background on ultragraphs, the edge shift space of an ultragraph, and ample groupoids and their topological full groups. 

In Section \ref{sec:groupoid} we associate a topological groupoid with an ultragraph that satisfies condition (RFUM). We show that the topology of the associated groupoid has a basis of compact open sets, so that the associated groupoid is ample. In Proposition \ref{prop:isoltedpoints} we characterize the isolated points of these groupoids and in Proposition \ref{prop:effectiveCharac} we describe when the groupoid is effective. We conclude Section~\ref{sec:groupoid} with a characterization of the elements of the topological full group of the groupoid associated with an ultragraph  (Proposition~ \ref{prop:characbisecandfullgrp}). 

In Section \ref{sec:MainResults} we prove our two main results, Theorem \ref{thm:MainIsoUGs} and Theorem \ref{main2}. To motivate our results, we give an example of an ultragraph satisfying the conditions of Theorem \ref{thm:MainIsoUGs}, and whose associated $C^*$-algebra cannot be realized as a graph $C^*$-algebra.



\section{Preliminaries}\label{sec1}

In this section we recall key definitions and set up notation regarding ultragraphs, groupoids and topological full groups. We start with ultragraphs. 

\subsection{Ultragraphs and the edge shift space} \label{sec:UltraGraphs}

Ultragraphs first appeared in  \cite{MR3600124}, as defined below. 

\begin{definicao}\label{def of ultragraph}
An \emph{ultragraph} is a quadruple $\mathcal{G}=(G^0, \mathcal{G}^1, r,s)$ consisting of two countable sets $G^0, \mathcal{G}^1$, a map $s:\mathcal{G}^1 \to G^0$, and a map $r:\mathcal{G}^1 \to P(G^0)\setminus \{\emptyset\}$, where $P(G^0)$ is the power set of $G^0$.
\end{definicao}

A key object when studying ultragraphs are generalized vertices, which we define below. 

\begin{definicao}\label{def of mathcal{G}^0}
Let $\mathcal{G}$ be an ultragraph. Define $\mathcal{G}^0$ to be the smallest subset of $P(G^0)$ that contains $\{v\}$ for all $v\in G^0$, contains $r(e)$ for all $e\in \mathcal{G}^1$, and is closed under finite unions and nonempty finite intersections. Elements of $\G^0$ are called \emph{generalized vertices}.
\end{definicao}




Next we set up notation that will be used throughout the paper. This agrees with notation introduced in \cite{MR2457327} and \cite{MR3600124}. 

Let $\mathcal{G}$ be an ultragraph. A \textit{finite path} in $\mathcal{G}$ is either an element of $\mathcal{G}%
^{0}$ or a sequence of edges $e_{1}\ldots e_{k}$ in $\mathcal{G}^{1}$ where
$s\left(  e_{i+1}\right)  \in r\left(  e_{i}\right)  $ for $1\leq i\leq k$. If
we write $\alpha=e_{1}\ldots e_{k}$, the length $\left|  \alpha\right|  $ of
$\alpha$ is  $k$. The length $|A|$ of a path $A\in\mathcal{G}^{0}$ is
zero. We define $r\left(  \alpha\right)  =r\left(  e_{k}\right)  $ and
$s\left(  \alpha\right)  =s\left(  e_{1}\right)  $. For $A\in\mathcal{G}^{0}$,
we set $r\left(  A\right)  =A=s\left(  A\right)  $. The set of
finite paths in $\mathcal{G}$ is denoted by $\mathcal{G}^{\ast}$. An \textit{infinite path} in $\mathcal{G}$ is an infinite sequence of edges $\gamma=e_{1}e_{2}\ldots$ in $\prod \mathcal{G}^{1}$, where
$s\left(  e_{i+1}\right)  \in r\left(  e_{i}\right)  $ for all $i$. The set of
infinite paths  in $\mathcal{G}$ is denoted by $\mathfrak
{p}^{\infty}$. The length $\left|  \gamma\right|  $ of $\gamma\in\mathfrak
{p}^{\infty}$ is defined to be $\infty$. A vertex $v$ in $G^0$ is
called a \emph{sink} if $\left|  s^{-1}\left(  v\right)  \right|  =0$, it is
called an \emph{infinite emitter} if $\left|  s^{-1}\left(  v\right)  \right|
=\infty$, and it is called a \emph{source} if $v\notin r(e)$ for all $e\in\GG^1$. For $v,w\in G^0$, we define $v\GG^1=\{e\in\GG^1\mid s(e)=v\}$, $\GG^1w=\{e\in\GG^1\mid w\in r(e)\}$ and $v\GG^1w=v\GG^1\cap \GG^1w$. 

For $n\geq1,$ we define
$\mathfrak{p}^{n}:=\{\left(  \alpha,A\right)  :\alpha\in\mathcal{G}^{\ast
},\left\vert \alpha\right\vert =n,$ $A\in\mathcal{G}^{0},A\subseteq r\left(
\alpha\right)  \}$. We specify that $\left(  \alpha,A\right)  =(\beta,B)$ if
and only if $\alpha=\beta$ and $A=B$. We set $\mathfrak{p}^{0}:=\mathcal{G}%
^{0}$ and we let $\mathfrak{p}:=\coprod\limits_{n\geq0}\mathfrak{p}^{n}$. We embed the set of finite paths $\GG^*$ in $\mathfrak{p}$ by sending $\alpha$ to $(\alpha, r(\alpha))$. We
define the length $\left\vert \left(
\alpha,A\right)  \right\vert $ of a pair $\left(  \alpha,A\right)  $  to be $\left\vert
\alpha\right\vert $. We call $\mathfrak{p}$ the \emph{ultrapath space}
associated with $\mathcal{G}$ and the elements of $\mathfrak{p}$ are called
\emph{ultrapaths}. Each $A\in\mathcal{G}^{0}$ is regarded as an ultrapath of length zero and can be identified with the pair $(A,A)$. We  extend the range map $r$ and the source map $s$ to
$\mathfrak{p}$ by declaring that $r\left(  \left(  \alpha,A\right)  \right)
=A$, $s\left(  \left(  \alpha,A\right)  \right)  =s\left(  \alpha\right)
$ and $r\left(  A\right)  =s\left(  A\right)  =A$.

We concatenate elements in $\mathfrak{p}$ in the following way: If $x=(\alpha,A)$ and $y=(\beta,B)$, with $|x|\geq 1, |y|\geq 1$, then $x\cdot y$ is defined if and only if
$s(\beta)\in A$, in which case, $$x\cdot y:=(\alpha\beta,B).$$ 
Also we
specify that:
\begin{equation}
x\cdot y=\left\{
\begin{array}
[c]{ll}%
x\cap y & \text{if }x,y\in\mathcal{G}^{0}\text{ and if }x\cap y\neq\emptyset\\
y & \text{if }x\in\mathcal{G}^{0}\text{, }\left|  y\right|  \geq1\text{, and
if }s\left(  y\right)\in x   \\
x_{y} & \text{if }y\in\mathcal{G}^{0}\text{, }\left|  x\right|  \geq1\text{,
and if }r\left(  x\right)  \cap y\neq\emptyset
\end{array}
\right.  \label{specify}%
\end{equation}
where, if $x=\left(  \alpha,A\right)  $, $\left|  \alpha\right|  \geq1$ and if
$y\in\mathcal{G}^{0}$, the expression $x_{y}$ is defined to be $\left(
\alpha,A\cap y\right)  $. Given $x,y\in\mathfrak{p}$, we say that $x$ has $y$ as an initial segment if
$x=y\cdot x^{\prime}$, for some $x^{\prime}\in\mathfrak{p}$, with $s\left(
x^{\prime}\right)  \cap r\left(  y\right)  \neq\emptyset$. 

We extend the source map $s$ to $\mathfrak
{p}^{\infty}$, by defining $s(\gamma)=s\left(  e_{1}\right)  $, where
$\gamma=e_{1}e_{2}\ldots$. We may concatenate pairs in $\mathfrak{p}$, with
infinite paths in $\mathfrak{p}^{\infty}$ as follows. If $y=\left(
\alpha,A\right)  \in\mathfrak{p}$, and if $\gamma=e_{1}e_{2}\ldots\in
\mathfrak{p}^{\infty}$ are such that $s\left(  \gamma\right)  \in r\left(
y\right)  =A$, then the expression $y\cdot\gamma$ is defined to be
$\alpha\gamma=\alpha e_{1}e_{2}...\in\mathfrak{p}^{\infty}$. If $y=$
$A\in\mathcal{G}^{0}$, we define $y\cdot\gamma=A\cdot\gamma=\gamma$ whenever
$s\left(  \gamma\right)  \in A$. Of course $y\cdot\gamma$ is not defined if
$s\left(  \gamma\right)  \notin r\left(  y\right)  =A$. 

\begin{remark} To simplify notation we omit the dot in the definition of concatenation, so that $x\cdot y$ will be denoted by $xy$.
\end{remark}

Given $\alpha,\beta\in \mathcal{G}^*$ we say $\alpha$ is an \textit{initial segment} of $\beta$, written as $\alpha<\beta$, if there is a $\gamma\in\mathcal{G}^*$ with $|\gamma|>0$ such that $\be=\al\gamma$. Similarly, for $\alpha\in \mathcal{G}^*$ and $\beta\in\mathfrak{p}^{\infty}$, in which case $\gamma\in \mathfrak{p}^{\infty}$. The paths $\al$ and $\be$ are \textit{disjoint} if they are different and neither one is an initial segment of the other. 

\begin{definicao} \label{dfn:ultraDisjoint}
    Two ultrapaths $(\alpha,A)$ and $(\beta,B)$ are \textit{disjoint} if one of the following conditions is satisfied:
    \begin{itemize}
        \item $\alpha$ and $\beta$ are disjoint paths;
        \item $\alpha=\beta$ and $A$ and $B$ are disjoint sets;
        \item $\alpha$ is a initial segment of $\beta$, say $\beta=\alpha\gamma$, and $s(\gamma)\notin A$;
        \item $\beta$ is a initial segment of $\alpha$, say $\alpha=\beta\gamma$, and $s(\gamma)\notin B$.
    \end{itemize}
\end{definicao}

\begin{definicao}
\label{infinte emitter} For each subset $A$ of $G^{0}$, let
$\varepsilon\left(  A\right)  $ be the set $\{ e\in\mathcal{G}^{1}:s\left(
e\right)  \in A\}$. We say that a set $A$ in $\mathcal{G}^{0}$ is an
\emph{infinite emitter} whenever $\varepsilon\left(  A\right)  $ is infinite.
\end{definicao}



The key concept in the definition of the shift space $X$ associated to an ultragraph without sinks $\GG$ is that of minimal infinite emitters. We recall this below.

\begin{definicao}\label{minimal} Let $\GG$ be an ultragraph and $A\in \GG^0$. We say that $A$ is a minimal infinite emitter if it is an infinite emitter that contains no proper subsets (in $\GG^0$) that are infinite emitters. 
For a finite path $\alpha$ in $\GG$, we say that $A$ is a minimal infinite emitter in $r(\alpha)$ if $A$ is a minimal infinite emitter and $A\subseteq r(\alpha)$. We denote the set of all minimal infinite emitters in $r(\alpha)$ by $M_\alpha$.
\end{definicao}





Associated to an ultragraph with no sinks, we have the topological space $X= \mathfrak{p}^{\infty} \cup X_{fin}$, where 
$$X_{fin} = \{(\alpha,A)\in \mathfrak{p}: |\alpha|\geq 1 \text{ and } A\in M_\alpha \}\cup
 \{(A,A)\in \GG^0: A \text{ is a minimal infinite emitter}\}, $$ and the topology has a basis given by the collection $$\{D_{(\beta,B)}: (\beta,B) \in \mathfrak{p}, |\beta|\geq 1\ \} \cup \{D_{(\beta, B),F}:(\beta, B) \in X_{fin}, F\subset \varepsilon\left( B \right), |F|<\infty \},$$ where for each $(\beta,B)\in \mathfrak{p}$ we have that $$D_{(\beta,B)}= \{(\beta, A): A\subset B \text{ and } A\in M_\beta \}\cup\{y \in X: y = \beta \gamma', s(\gamma')\in B\},$$ and, for $(\beta,B)\in X_{fin}$ and $F$ a finite subset of $\varepsilon\left( B \right)$,  $$D_{(\beta, B),F}=  \{(\beta, B)\}\cup\{y \in X: y = \beta \gamma', \gamma_1' \in \ \varepsilon\left( B \right)\setminus F\}.$$
 
For $(\beta,B)\in \mathfrak{p}$ and $F\subseteq \varepsilon\left( B \right)$ finite, define $$D_{(\beta,B),F}= \{(\beta, A): A\subset B \text{ and } A\in M_\beta \}\cup\{y \in X: y = \beta \gamma',\gamma_1' \in \ \varepsilon\left( B \right)\setminus F\},$$

Notice that, for $(\beta,B)\in \mathfrak{p}$, $D_{(\beta,B),F}=D_{(\beta,B)}\setminus \left(\bigcup_{\gamma_1\in F}D(\beta\gamma_1,r(\beta\gamma_1)\right)$, which is closed because it is a difference of a closed set \cite{MR3938320} and a union of open sets. Also $D_{(\beta,B),F}=\bigcup_{A\in M_{\beta}}D_{(\beta,A),F\cap\varepsilon(A)}\cup\bigcup_{\gamma_1\in\varepsilon(B)\setminus F}D(\beta\gamma_1,r(\beta\gamma_1))$, which is an open set.
 

\begin{remark}
Note that if two ultrapaths $(\alpha,A)$ and $(\beta,B)$ are disjoint then the sets $D_{(\alpha,A)}$ and $D_{(\beta,B)}$ are disjoint.
\end{remark}

We recall below Condition~(RFUM), which guarantees that the (R)ange of each edge is a (F)inite (U)nion of (M)inimal infinite emitters and single vertices.


{\bf Condition (RFUM):} For each edge $e\in \GG^1$ its range can be written as $$r(e) = \displaystyle \bigcup_{n=1}^k A_n,$$ where $A_n$ is either a minimal infinite emitter or a single vertex. 

It was shown in \cite{MR3938320} that under Condition~(RFUM), the shift space $X$ has a basis of open, compact sets. If condition (RFUM) is not satisfied, then certain basic open sets may not be compact (c.f. \cite[Remark 3.10]{MR3938320}). Having a basis of compact open sets such that each compact open set can be expressed as a disjoint union of these basic sets (Lemma \ref{lem:disjointunionofcylindersets}) is a crucial feature in our work, and we therefore make the following assumptions for the remainder of this paper: {\bf all ultragraphs  are assumed to have no sinks and satisfy Condition~(RFUM)}.

\begin{remark}\label{cylindersets}
    The basis considered in \cite{MR3938320} does not included the sets $D_{(\beta,B),F}$ when $(\beta,B)\notin X_{fin}$ and $F\neq\emptyset$. However as seen above they are open and closed (and thus compact under condition (RFUM)). We will include these sets in the basis in order to simplify some of the proofs in this paper.
\end{remark}

We associate to the space $X$ a shift map:
\begin{definicao}\label{shift-map-def}
The \emph{shift map} is the function $\sigma : X\setminus \G^{(0)} \rightarrow X$ defined by $$\sigma(x) =  \begin{cases} \gamma_2 \gamma_3 \ldots & \text{ if $x = \gamma_1 \gamma_2 \ldots \in \mathfrak{p}^\infty$} \\ (\gamma_2 \ldots \gamma_n,A) & \text{ if $x = (\gamma_1 \ldots \gamma_n,A) \in X_{fin}$ and $|x|> 1$} \\(A,A) & \text{ if $x = (\gamma_1,A) \in X_{fin}.$} 
\end{cases}$$
We call $X$ together with the shift map the \emph{edge shift space}.
\end{definicao}

\begin{remark}
Notice that we do not define the shift map for elements of $X$ of length zero, differently to what is done in \cite{MR3600124, MR3938320, MR3918205,gonccalves2018li, gonccalves2019ultragraph}, 
since the shift map may fail to be continuous on paths of length zero. This ensures that it has all the `nice' continuity properties that we require on paths of length greater than zero to build an \' etale groupoid from the shift (\cite[Proposition 3.16]{MR3600124}). 
\end{remark}

\subsection{Groupoids and their topological full groups}\label{subsec:groupoid.full.group}

In this section we gather necessary definitions and background on the topological full group of an ample groupoid, closely following \cite{MR3950815}.

 A \emph{groupoid} $G$ is a small category of isomorphisms. A \emph{topological groupoid} is a groupoid equipped with a topology making the operations of multiplication and taking inverse continuous. The elements of the form $gg^{-1}$ are called \emph{units}. We denote the set of units of  $G$ by $G^{(0)}$, and refer to $G^{(0)}$ as the \emph{unit space}. We think of the unit space as a topological space equipped with the relative topology from $G$. The \emph{source} and \emph{range} maps are given by $s(g)=g^{-1}g$ and $r(g)= gg^{-1}$, for $g\in G$. These maps are necessarily continuous when $G$ is a topological groupoid.

An \emph{\'etale groupoid} is a topological groupoid  $  G$ such that its unit space  $G\skel 0$ is locally compact and Hausdorff and its range map is a local homeomorphism (this implies that the source map and the multiplication map are also local homeomorphisms). A bisection of $G$ is a subset $B\subseteq  G$ such that the restriction of the range and source maps to $B$ are injective. A bisection $U$ is called full if we have  $s(U) = r(U) = G^{(0)}$.  An \'etale groupoid is \emph{ample} if its unit space has a basis of compact open sets or, equivalently, if the arrow space $ G$ has a basis of compact open bisections.

The \emph{isotropy group} of a unit $x\in G^{(0)}$ is the group $G_x^x = \{g\in \GG \mid s(g)=r(g)=x\}$, and the \emph{isotropy bundle} is
\[G'= \{g\in G \mid s(g)=r(g)\} = \bigsqcup_{x \in \G^{(0)}} G_x^x.\]
We say that $G$ is \emph{effective} if the interior of $G'$ equals $G^{(0)}$. We call $G$ \emph{topologically principal} if the set of points in $G^{(0)}$ with trivial isotropy group are dense in $G^{(0)}$.

If $x\in G^{(0)}$, then the \textit{orbit} of $x$ is defined by $$\mathrm{Orb}_{G}(x)=\{y\in G^{(0)}: \text{ there exists } g\in G \text{ with } s(g)=x, r(g)=y\}.$$ 
A subset $A \subset G^{(0)}$ is \textit{wandering} if $|A \cap \mathrm{Orb}_{G}(x)| = 1$ for all $x \in A$. We say that $G$ is\textit{ non-wandering} if $G^{(0)}$ has no non-empty clopen wandering subsets.

To each bisection $U\subseteq \GG$ in an \'etale groupoid we associate a homeomorphism \[\pi_U \colon s(U)\to r(U)\] given by $r_{\vert U} \circ (s_{\vert U})^{-1}$. Whenever $U$ is a full bisection, $\pi_U$ is a homeomorphism of $G^{(0)}$. 

For a topological space $X$ we denote the group of self-homeomorphisms of $X$ by $\homeo(X)$. By an \emph{involution} we mean a homeomorphism (or more generally, a group element) $\phi$ with $\phi^2 = \id_X$.  For a homeomorphism~$\phi \in \homeo(X)$, we define the \emph{support of $\phi$} to be the (regular) closed set~$\overline{\{x\in X \ \vert \ \phi(x)\neq x\}}$, and denote it by $\supp(\phi)$. We also define $\homeo_c(X)=\{\phi\in\homeo(X)\mid\supp(\phi)\text{ is compact open}\}$.

\begin{definicao}\label{def:tfg}
Let $G$ be an effective ample groupoid. The \emph{topological full group} of $G$, denoted $\llbracket G \rrbracket$, is the subgroup of $\homeo\left(G^{(0)}\right)$ consisting of all homeomorphisms of the form $\pi_U$, where $U$ is a full bisection in $G$ such that $\supp(\pi_U)$ is compact. We will denote by~$\DD(\llbracket G \rrbracket)$ its commutator subgroup.
\end{definicao}

In the topological full group, composition and inversion of the homeomorphisms correspond to multiplication and inversion of the bisections, that is, 
\begin{itemize}
\item $\pi_{\mathcal{G}^{(0)}} = \id_{\mathcal{G}^{(0)}} = 1$
\item $\pi_U \circ \pi_V = \pi_{UV}$
\item $\left(\pi_U\right)^{-1} = \pi_{U^{-1}}$
\end{itemize}

\begin{lema}\cite[Lemma 3.7]{MR3950815} \label{bisectiondecomposition}
Let $G$ be an effective ample groupoid, and let $\pi_U\in \llbracket G \rrbracket$. Then we have a decomposition 
$$ U=U^\bot \sqcup \left(G^{(0)}\backslash \supp(\pi_U)  \right),$$
where $U^\bot$ is a compact open bisection with $s(U^\bot)=r(U^\bot)=\supp(\pi_U)$. Conversely, any compact bisection $V\subseteq G$ with $s(V)=r(V)$ defines an element $\pi_{\tilde{V}} \in\llbracket  G\rrbracket$ with $\supp(\pi_{\tilde{V}})\subseteq s(V)$ by setting $\tilde{V}=V\sqcup (G^{(0)}\backslash s(V))$.
\end{lema}

\begin{lema}\cite[Lemma 3.8]{MR3950815} \label{NonEmptyIntersecDfnsInvolution}
Let $G$ be an effective ample groupoid. Any compact bisection $V\subseteq G$
which satisfies $s(V )\cap r(V ) = \emptyset$ defines an involution $\pi_{\hat{V}}\in 
\llbracket G \rrbracket$ by setting $\hat{V}$ equal to $ V \sqcup V^{-1} \sqcup (G^{(0)}\backslash (s(V ) \cup r(V )))$. Moreover, $\supp(\pi_{\hat{V}} ) \subseteq s(V) \cup r(V)$.
\end{lema}

\begin{definicao}\cite[Definitions 5.1 and 6.1]{MR3950815}\label{dfn:SpaceGroupPairsFaithful}
A \textit{space-group pair} consists of a pair $(\Gamma,X)$, where X is a Hausdorff space with a basis of compact open sets, and $\Gamma$ is a subgroup of  $\mathrm{Homeo}_c(X)$.  A class K of space-group pairs is called \textrm{faithful} if for every $(\Gamma_1,X_1), (\Gamma_2,X_2)\in K$ and every every group isomorphism $\Phi : \Gamma_1 \to \Gamma_2$ there is a homeomorphism $\phi : X_1 \to X_2$ such that $\Phi(\gamma) = \phi \circ\gamma\circ\phi^{-1}$  for every $\gamma\in\Gamma_1$.
\end{definicao}

\section{Ultragraph groupoids and their topological full groups}\label{sec:groupoid}

In this section we define an ample groupoid associated with an ultragraph with no sinks and satisfying condition (RFUM). We use this groupoid to extend some known results about graphs to the general setting of ultragraphs. In Proposition \ref{prop:characbisecandfullgrp} we characterize  the open bisections and describe the elements of the topological full group of the groupoid associated with an ultragraph, analogous to that of graphs (\cite{MR3950815}). This section is based on and extends \cite[Section 9]{MR3950815} to groupoids of ultragraphs.

Throughout this section we fix an ultragraph $\GG$ which satisfies condition (RFUM) and has no sinks. Let $(X,\sigma)$  denote the edge shift space associated to $\GG$ (see Definition~\ref{shift-map-def}).

We begin by describing the groupoid associated to an ultragraph and its topology. For any non-zero $m\in\N$ we let  $X^{\geq m}=\{y\in X\mid |y|\geq m\}$. Define 
$$\Gs:= \{(x,m-n,y)\in X\times \Z\times X\mid x\in X^{\geq m}, y\in X^{\geq n}, \sigma^m(x)=\sigma^n(y) \}. $$
The set of composable pairs is given by  \[\Gs^2=\{(x,m,y),(x^{\prime},n,y^{\prime})\in\Gs\mid y=x^{\prime}\}.\]
Then $\Gs$ is a groupoid with composition and involution given by 
\begin{equation*}
\begin{split}
    &(x,m,y)(y,n,y):=(x,mn,z), \text{ and } \\
    & (x,m,y)^{-1}:=(y,-m,x),
\end{split}
\end{equation*}
respectively. The unit space $\Gs^{(0)}$ of $\Gs$ is identified with $X$. To get a topology on $\Gs$ we define
$$ Z(U,m,n,V):=\{(x,m-n,y)\in \Gs \mid x\in U,y\in V, \sigma^m(x)=\sigma^n(y)\}, $$
where $U\subseteq X^{\geq m} $ and $V\subseteq X^{\geq n}$ are open sets such that $\sigma^m|_U$ and $\sigma^n|_V$ are injective and $\sigma^m(U)=\sigma^n(V)$. The sets $Z(U,m,n,V)$, ranging over $U$ and $V$ that satisfy these conditions, form a basis for a locally compact Hausdorff topology on $\Gs$. This topology is a direct analogue of the topology for the boundary path groupoid of directed graphs, which is well-known to be an \'etale groupoid (see for example \cite{MR1770333}). 

\begin{remark} Notice that the C*-algebra associated to the groupoid defined above coincides with the usual ultragraph C*-algebra, see \cite{phdthesisFelipe, tascagoncalves} for details. 
\end{remark}

\begin{remark}
We point out that the description of the usual ultragraph C*-algebra using groupoids was done first in \cite{MR2457327} by Marrero and Muhly. In order to define the unit space of their groupoid, we need to work with filters and ultrafilters. Condition (RFUM) allows us to define the unit space in simpler terms using only the notion of minimal infinite emitters as in Section \ref{sec1}.

Also, in order to work with the full group, we also need a suitable basis of open bisections of the groupoid. The basis considered just above \cite[Lemma 21]{MR2457327} is weaker than the one needed to obtain the C*-algebra of an ultragraph. Comparing their basis with ours in Lemma \ref{lem:thirdbasisfortop}, they only consider the bisections where $F_A$ is the empty set. However if we only considering $F_A$ to be the empty set for a usual graph with an infinite emitter results in a non-Hausdorff topology.

With minor adjustments, the main results of \cite{MR2457327} still hold. See for example \cite{Gil3}, where the class of C*-algebras of labelled spaces, which include the class of C*-algebras of ultragraphs, is described in terms of groupoids using an approach similar to \cite{MR2457327}.
\end{remark}

We aim to characterize the topology on $\Gs$ in terms of a different basis, which will be helpful to show that $\Gs$ is an ample groupoid. For this we use the cylinder sets defined in Section \ref{sec:UltraGraphs} and the following lemmas. 

\begin{lema}\label{lem:sameunitsameMIE}
Let $\alpha,\beta\in \GG^*$. Assume that $A\in\GG^0$,  $(\alpha,A), (\beta,A)\in\mathfrak{p}$ and $C\subseteq A$. Then $C\in M_\alpha$ if and only if $C\in M_\beta$.
\end{lema}
\begin{proof}
If $C\in M_\alpha\backslash M_\beta$, then $C$ is an infinite emitter. Since $C \notin M_\beta$, there exists an infinite emitter $D$ in $\G^0$ such that $D\subsetneq C$. Hence C is not minimal (in $M_\alpha$), a contradiction. The same argument, with the roles of  $M_\alpha$ and $M_\beta$ reversed, proves the converse.
\end{proof}

\begin{lema}\label{lem:EqualShiftOnCylinders}
Assume that $(\alpha,A),(\beta,B)\in \mathfrak{p}$. Then $\sigma^{|\alpha|}(D_{(\alpha,A)})=\sigma^{|\beta|}(D_{(\beta,B)})$ if and only if $A=B$.
\end{lema}

\begin{proof}
Assume that $\sigma^{|\alpha|}(D_{(\alpha,A)})=\sigma^{|\beta|}(D_{(\beta,B)})$, but $A\neq B$. Without loss of generality we may assume that there exists a vertex $v\in B\backslash A$. Since $\GG$ has no sinks, there is an element $\gamma\in X^{\geq1}$ such that $s(\gamma) =v $ and $\beta\gamma\in D_{(\beta,B)}$. Then $\gamma\in\sigma^{|\beta|}(D_{(\beta,B)}) =\sigma^{|\alpha|}(D_{(\alpha,A)})$, which implies that $\alpha\gamma\in D_{(\alpha,A)}$. However, since $s(\gamma)=v\notin A$, we have a contradiction. Hence $A=B$. 

For the converse assume that $A=B$ and let $\gamma\in \sigma^{|\alpha|}(D_{(\alpha,A)})$. Then there exists $x\in X$ such that $x=\alpha\gamma$. Either $\gamma\in X_{fin}$ or $\gamma\in \mathfrak{p}^{\infty}$. We consider these cases separately. 

First assume that $\gamma\in X_{fin}$. That is, $\gamma = (\gamma,C)$ with $s(\gamma)\in A$ and $C\in M_\gamma$. If $|\gamma|=0$ then $\gamma=(C,C)$ and $C\subset A$ and $C\in M_\alpha$. Then $C\in M_\beta$ by Lemma \ref{lem:sameunitsameMIE}. Hence $(\beta,C)\in D_{(\beta,A)}$ and $\sigma^{|\beta|}(\beta,C)=(C,C)$, which shows that $ \gamma\in \sigma^{|\beta|}(D_{(\beta,B)})$. If $|\gamma|\geq 1$, then $s(\gamma)\in A=B$ and $C\in M_\gamma$, which implies that $(\beta\gamma, C)\in D_{(\beta,A)}$. Then $(\gamma, C)=\sigma^{|\beta|}(\beta\gamma,C)\in \sigma^{|\beta|}(D_{(\beta,B)})$.
For the second case, assume that $\gamma\in \mathfrak{p}^{\infty}$. Then $s(\gamma)\in A=B$. Thus $\beta\gamma\in\mathfrak{p}^{\infty}$, which implies that $\sigma^{|\beta|}(\beta\gamma)=\gamma\in \sigma^{|\beta|}(D_{(\beta,B)})$.

The above shows that $\sigma^{|\alpha|}(D_{(\alpha,A)})\subseteq \sigma^{|\beta|}(D_{(\beta,B)})$, and the same argument with with the assumption that $\gamma\in \sigma^{|\beta|}(D_{(\beta,B)})$ gives the reverse inclusion. Hence  $\sigma^{|\alpha|}(D_{(\alpha,A)})=\sigma^{|\beta|}(D_{(\beta,B)})$, completing the proof.
\end{proof}

We now give an alternative description of the topology on $\Gs$ in terms of the cylinder sets that define the topology on $X$. Let $\alpha,\beta\in \GG^*$, $A\in \GG^0$ such that $(\alpha,A), (\beta,A)\in \mathfrak{p}$ and let $U,V\subset X$. We define
\begin{equation*}
\mathcal{Z}(U,\alpha,A,\beta,V):=\{(x,|\alpha|-|\beta|,y)\in \Gs \mid x\in U, y\in V, \sigma^{|\alpha|}(x)=\sigma^{|\beta|}(y)\in D_{(A,A)}\}.
\end{equation*}

\begin{lema}\label{lem:secondbasisfortop}
Let $$\mathscr{C}=\{\mathcal{Z}(U,\alpha,A,\beta,V)\mid \sigma^{|\alpha|}(U)=\sigma^{|\beta|}(V)\}, $$
parametrized over all $\alpha,\beta\in \GG^*$ and $A\in \GG^0$ such that $(\alpha,A), (\beta,A)\in \mathfrak{p}$, and with  $U\subseteq D_{(\alpha,A)}$ and $V\subseteq D_{(\beta,A)}$ compact open sets. Then $\mathscr{C}$ is a basis of compact open sets for the topology on $\Gs$ and hence $\Gs$ is ample.
\end{lema}
\begin{proof}
We show that $\mathscr{C}$ is a collection of open subsets such that for any open set $W\subseteq \Gs$, and each $w\in W$, there is a $C\in\mathscr{C}$ such that $w\in C\subseteq W$.

Let $\mathcal{Z}(U,\alpha,A,\beta,V)\in \mathscr{C}$. Then since $D_{(\alpha,A)} \subseteq X^{\geq|\alpha|}$ and $D_{(\beta,A)} \subseteq X^{\geq|\beta|}$ are open it follows that $U\subseteq D_{(\alpha,A)} \subseteq X^{\geq|\alpha|}$ and $V\subseteq D_{(\beta,A)} \subseteq X^{\geq|\beta|}$ are open. It is clear that $\sigma^{|\alpha|}|_U$ and $\sigma^{|\beta|}|_V$ are injective, since we shift elements of the form $\alpha x\in D_{(\alpha,A)}$ by the length of $|\alpha|$ (and similarly for $\beta$). Also, $\sigma^{|\alpha|}(U)=\sigma^{|\beta|}(V)$ follows by definition. Hence $\mathcal{Z}(U,\alpha,A,\beta,V)=Z(U,|\alpha|,|\beta|,V)$  is an open set in the topology of $\Gs$, which shows that $\mathscr{C}$ consists of a collection of open sets. 

Let $W\subseteq \Gs$ be an open set with $(x,k,y)\in W$. Then there exists a basis element $ Z(U,m,n,V)$ of the topology on $\Gs$ such that $(x,k,y)\in  Z(U,m,n,V)\subseteq W$. Then $k=m-n$ for some $m,n\in\N$ with $\sigma^m(x)=\sigma^n(y)\in X$. Let $z:=\sigma^m(x)=\sigma^n(y)$. Hence there exists $\alpha, \beta\in\GG^*$ such that $|\alpha|=m,|\beta|=n$ and $\alpha z, \beta z\in X$. Note that $r(\alpha)\cap r(\beta)\neq \emptyset$, since $s(z)\in r(\alpha)\cap r(\beta)\neq \emptyset$. Then  $C=\mathcal{Z}(U,\alpha,r(\alpha)\cap r(\beta),\beta,V) \subseteq  Z(U,m,n,V)\subseteq W$. 

Finally we show that each $\mathcal{Z}(U,\alpha,A,\beta,V)\in \mathscr{C}$ is a compact set. To see this, take any net $\{(x_i,k_i,y_i)\}$ in $\mathcal{Z}(U,\alpha,A,\beta,V)\in \mathscr{C}$. Then the first and third coordinates will each have a convergent subnet due to the compactness of $U$ and $V$, say to $x$ and $y$, and the middle coordinate stays constant for all $i$. Thus the net has a subnet converging to $\{(x,|\alpha|-|\beta|,y)\}$, showing compactness. Now, since $\Gs$ is an \' etale groupoid with a basis of compact open sets, it follows that $\Gs$ is ample.
\end{proof}

Next we extend some known results of graphs to ultragraphs that relate the properties of the graph with certain topological properties of its associated groupoid. We begin with some definitions. 
A finite path $\al\in \GG^*$ with $|\al|>0$ is a \textit{loop} if $s(\al)\in r(\al)$. We say $\al$ is a loop based at $A\in \GG^0$ if $s(\al)\in A$. If $\al=\al_1\ldots\al_n$ is a loop, then $\al$  is a \textit{simple loop} if $s(\al_i)\neq s(\al_1)$ for  $i\neq 1$; that is, the loop does not pass through $s(\al_1)$ multiple times. An \textit{exit} for a loop $\al=\al_1\ldots\al_n$ is either of the following:
\begin{enumerate}
    \item an edge $e\in \GG^1$ such that there exists an $i$ for which $s(e)\in r(\al_i)$,  but $e\neq  \al_{i+1}$,
    \item a sink $w$ such that $w \in r(\al_i)$ for some $i$.
\end{enumerate}

\begin{remark}
Since we are working only with ultragraphs that have no sinks, in this paper an exit for a loop takes only form 1. above. So for a simple loop $\alpha_1\ldots \alpha_n$ with no exists we have that $r(\alpha_i)=\{s(\alpha_{i+1})\}$ for all $i$, with the convention that $\alpha_{n+1}=\alpha_1$.
\end{remark}

An ultragraph $\GG$ satisfies \textit{Condition (L)} if every loop has an exit. An infinite path $x\in\mathfrak{p}^{\infty}$ is called \textit{eventually periodic} if $x=\al\gamma^{\infty}$ for some finite path $\al$ and some loop $\gamma$, where $\gamma^\infty$ denotes the infinite path $\gamma \gamma \gamma \ldots$. An infinite path $x\in\mathfrak{p}^{\infty}$ is called \textit{wandering} if $ |\{i\in\N\mid s(x_i)=v\}|<\infty$ for every $v\in G^0$. An infinite wandering path $x\in\mathfrak{p}^{\infty}$ has a \textit{semi-tail} if $|\{i\in\N: |\varepsilon(s(x_i))|>1\text{ or }|r(x_i)|>1\}|<\infty$.

Note that if an infinite path $\alpha_1 \alpha_2 \ldots$ in a graph is wandering and eventually each vertex emits only one edge (that is, there is $N$ such that $|s^{-1}(s(\alpha_i))|=1$ for all $i>N$), then by definition the path has a semi-tail, and hence it is an isolated point. This need not be the case in ultragraphs: suppose $\al=\al_1\ldots\in\mathfrak{p}^{\infty}$ is such that $r(\al_i)=\{s(\al_{i+1}),v_i\}$ for distinct vertices $s(\al_1),v_1,s(\al_2),v_2\ldots$. Then $|\{j\in\N \mid s(\al_j)=s(\al_i)\}|=1$ for every $i\in\N$ and $\al$ is wandering. However
$\alpha$ does not have a semi-tail, because $ |r(\alpha_i)|>1$ for every $i\in\N$, and is therefore not an isolated point, because we can deviate from the path at any $r(\alpha_i)$.

 The next proposition characterizes all the isolated points in $X$.  

\begin{proposicao}\label{prop:isoltedpoints}
Let $\GG$ be an ultragraph with no sinks that satisfies condition (RFUM) and let $X$ be its associated edge shift space. 
\begin{enumerate}
    \item An eventually periodic path $x=\al\gamma^{\infty}\in\mathfrak{p}^{\infty}$ is an isolated point if and only if $\gamma$ does not have an exit.
    \item A wandering path $x\in\mathfrak{p}^{\infty}$ is isolated point if and only if $x$ has  a semi-tail. 
\end{enumerate}
These are the only isolated points.
\end{proposicao}
\begin{proof}
1. Assume that $x=\al\gamma^{\infty}\in\mathfrak{p}^{\infty}$ is isolated, but that $\gamma$ has an exit. Since $\GG$ has no sinks there is an edge $e\in\GG^1$ such that $s(e)\in r(\gamma_i)$ for some $0<i\leq |\gamma|$ and $e\neq \gamma_{i+1}$. Let $D_{(\be,A)F}$ be any basic open neighborhood of $x$. Then $\be<x$   and thus there is $n\in\N$ such that $\be<\al\gamma^n$. Let $y=\al\gamma^n\gamma_1\ldots\gamma_i e$. Then $D_{(y,r(e))}\subset D_{(\be,A)F}$ and contains a point distinct from $x$. Hence every neighborhood of $x$ contains a point different from $x$, which contradicts that $\{x\}$ is open.

For the converse, assume that $x=\al\gamma^{\infty}\in\mathfrak{p}^{\infty}$ and that $\gamma$ has no exits. Since $\GG$ has no sinks and $\gamma$ has no exits, it follows that $|r(\gamma_i)|=1$, for every $i=1,\ldots,|\gamma|$. Therefore $D_{(\al\gamma,r(\gamma))}$ only contains infinite paths in $X$. Moreover, there is exactly one infinite path, namely $x$, implying that $\{x\}=D_{(\al\gamma,r(\gamma))}$ and is thus open. 

2. Let $x\in\mathfrak{p}^{\infty}$ be a wandering isolated point. Let $D_{(\be,A)F}$ be any basic open neighborhood of $x$. Hence $\be<x$.  If $x$ does not have a semi-tail, then  there is an $i>|\be|+1$ such that $|r(x_i)|>1$ or $ |\varepsilon(s(x_i))|>1$. If $|r(x_i)|>1$ then, since $\GG$ has no sinks, there is an edge $e\in\GG^1$, $e\neq x_{i+1}$, such that $s(e)\in r(x_i)$. Then there is an infinite path (distinct from $x$) $x_1\ldots x_i e \ldots \in D_{(\be,A)F}$. 
If $ |\varepsilon(s(x_i))|>1$ then there is an edge $e\neq x_i$ such that $s(e)=s(x_i)$ and hence there is an infinite path  $x_1\ldots x_{i-1} e \ldots \in D_{(\be,A)F}$ (distinct from $x$).
Therefore every neighborhood of $x$ contains a point different from $x$, which contradicts that $\{x\}$ is open.

Conversely, assume that $x\in\mathfrak{p}^{\infty}$ is wandering and has a semi-tail. Then there is $n_0\in\N$ such that $ |\varepsilon(s(x_n))|=1$ and $|r(x_n)|=1$ for all $n>n_0$. Hence $D_{(x_1\ldots x_{n+1}, r(x_{n+1}))}=\{x\}$ for all $n>n_0$ and thus $\{x\}$ is open.

 Finally, we show that these are only types of isolated points. Let $x$ be an isolated point. We claim that $x\in\mathfrak{p}^{\infty}$. To see this claim, suppose that $x$ is not an infinite path and let $D_{(\al,A)F}$ be any neighborhood of $x$. Then, since we do not have any sinks, we can extend $x$ to an infinite path $\tilde{x}=\tilde{x}_1\tilde{x}_2\ldots$ such that $\tilde{x}_{|\al|+1}\notin F$. Then $\tilde{x}\in D_{(\al,A)F}$, contradicting the fact that  $x$ is an isolated point. Hence $x$ must be an infinite path if it is an isolated point. If $x=x_1x_2\ldots \in\mathfrak{p}^{\infty}$ is neither eventually periodic nor wandering, then there exists $v\in G^0$ such that $|\{i\in\mathbb{N}: s(x_i)=v\}|=\infty$. Let $D_{(\alpha,A)F}$ be any neighborhood of $x$, and let $m,n$ be any indices such that $m<n$, $s(x_m)=s(x_n)=v$ and  $|\alpha|<|x_1\cdots x_m|$. Note that $\gamma:=x_{m}\cdots x_{n-1}$ is a loop. Then $x$ and $y:=x_1\cdots x_{m-1}\gamma^{\infty} $ are two distinct points and both are contained in $D_{(\alpha,A)F}$. However, this contradicts that $x$ is an isolated point, since  $D_{(\alpha,A)F}$ is arbitrary.
\end{proof}


\begin{proposicao} \label{prop:effectiveCharac}
Let $\GG$ be an ultragraph with no sinks and that satisfies condition (RFUM). Then the following are equivalent: 
\begin{enumerate}
    \item The groupoid $\Gs$ is effective.
    \item The ultragraph $\GG$ satisfies condition (L).
    \item The set of all elements in $X$ which are not eventually periodic is dense in $X$.
    \item The groupoid $\Gs$ is topologically principal.
\end{enumerate}   
\end{proposicao}
\begin{proof}
$(1)\Rightarrow(2)$: Assume that $\Gs$ does not satisfy Condition (L). Let $\gamma$ be a loop without an exit. By Proposition \ref{prop:isoltedpoints}(1), $\gamma^{\infty}$ is an isolated point. Then $|r(\gamma_i)|=1$ for $i=1,\ldots,|\gamma|$. Hence, for all $n\in\N$, the only point in $D_{(\gamma^n,r(\gamma))}$ is $\gamma^{\infty}$. Thus $\mathcal{Z}(D_{(\gamma^2,r(\gamma))},\gamma^2,r(\gamma),\gamma,D_{(\gamma,r(\gamma))})=\{(\gamma^{\infty},|\gamma|,\gamma^{\infty})\}$ is an open subset of the stabilizer subgroup at $(\gamma^{\infty},0,\gamma^{\infty})$, which is not contained in $\Gs^{(0)}$ and thus $\Gs$ is not effective.

$(2)\Rightarrow (3)$: Assume that $\GG$ satisfies Condition (L). Let $\gamma^{\infty}\in X$ and $D_{(\al,A)F}$ any neighborhood of $\gamma^{\infty}$. Then there exist $e\in\GG^1$ such that $s(e)\in r(\gamma_i)$ for some $0\leq i \leq|\gamma|$ and also $n\in\N$ such that  $D_{(\gamma^{n}\gamma_1\ldots\gamma_i e,r(e))}\subset D_{(\al,A)F}$. Let $y\in D_{(\gamma^{n}\gamma_1\ldots\gamma_i e,r(e))}$. Then $y$ is distinct from $x$. If $y$ is also eventually periodic, then we repeat the process above of taking an exit and forming a new neighborhood contained in  $D_{(\al,A)F}$. In this way we obtain an infinite path that is not eventually periodic contained in $D_{(\al,A)F}$. Hence the neighborhood $D_{(\al,A)F}$ of $\gamma^{\infty}$ has non-empty intersection with the subset of $X$ consisting of points which are not eventually periodic, showing that this set is dense in $X$.

$(3)\Rightarrow (4)$: Assume the set of elements in $X$ which are not eventually periodic is dense in $X$. Suppose that $(x,0,x)\in\Gs^{(0)}$ has non-trivial isotropy and let $U\subset\Gs^{(0)}$ be an open neighborhood of $(x,0,x)$. Then there is a $(y,0,y)\in U$ such that $y$ in infinite and not eventually periodic. Since $(y,0,y)\in\Gs^{(0)}$ has non-trivial isotropy if and only if $x$ is eventually periodic and infinite, it follows that $(y,0,y)$ has trivial isotropy, which shows that $\Gs$ is topologically principal.

$(4)\Rightarrow (1)$: This is a general fact for locally compact Hausdorff groupoids \cite{MR3189105}. 
\end{proof}

Let $\alpha,\beta\in \GG^*$ and $A\in \GG^0$ such that $(\alpha,A),(\beta,A)\in\mathfrak{p}$ and let $F\subseteq \varepsilon(A)$ be finite . Define 
\begin{eqnarray*}
\mathcal{Z}(\alpha,\beta,A,F_A)&:=& \mathcal{Z}(D_{(\alpha,A)F},\alpha,A,\beta,D_{(\beta,A)F}) \\
&=&\{(\alpha \xi, |\alpha|-|\beta|, \beta  \xi): \xi \in D_{(A,A)F}\}. 
\end{eqnarray*}
Note that all the sets  $\mathcal{Z}(\alpha,\beta,A,F_A)$ are compact open, since they are in $\mathscr{C}$. 

Our next goal is to prove that the sets $\mathcal{Z}(\alpha,\beta,A, F_A)$ also forms basis for the topology on $\Gs$ (Lemma \ref{lem:thirdbasisfortop}). This basis will allow us to characterize the open bisections and the elements of the topological full group of $\Gs$ (Proposition \ref{prop:characbisecandfullgrp}).

\begin{lema}\label{lem:cilynderIntersections}
Let $\al,\ap,\in\GG^*$ and $A,B\in\GG^0$ be such that  $(\al,A),  (\ap,B),  \in \mathfrak{p}$, and let $F_A\subseteq \varepsilon(A)$ and $F_B\subseteq \varepsilon(B)$ be finite subsets. Then $D_{(\al,A)F_{A}}\cap D_{(\ap,B)F_B}$ equals either 
\begin{enumerate}
    \item $\emptyset$, or
    \item $D_{(\al,A)F_A}$, or 
    \item $D_{(\ap,B)F_B}$, or 
    \item $D_{(\al,A\cap B)F_A\cup F_B}$.
\end{enumerate}
In addition, if $\al=\ap$, then 
\[ D_{(\al,A)F_A}\cup D_{(\ap,B)F_B} = D_{(\al,A\cup B)F_A\cap F_B}\]
\end{lema}

\begin{proof}
We have the following four cases:

(i) If $\al=\ap$, then $D_{(\al,A)F_A}\cap D_{(\al,B)F_B}=D_{(\al,A\cap B)(F_A\cup F_B)}$. In this case we also have that $D_{(\al,A)F_A}\cup D_{(\al,B)F_B}=D_{(\al,A\cup B)(F_A\cap F_B)}$.  Note that in this case if $(\al,A), (\ap,B)\in X_{fin}$  then $A=B$, since $A,B\in M_\alpha$ (Lemma \ref{lem:sameunitsameMIE}). 

(ii) Suppose $\al<\ap$, then $\ap=\al\gamma$ with $|\gamma|>0$ and such that $\gamma_1\in \varepsilon(A)\backslash F_A$. Then $D_{(\ap,B)F_B}\subseteq D_{(\al,A)F_A}$. Hence, if $\ap_{|\al|+1}\in \varepsilon(A)\backslash F_A$ then $D_{(\ap,B)F_B}\cap D_{(\al,A)F_A} = D_{(\ap,B)F_B}$ and empty otherwise. 

(iii) By the same reasoning as in (ii), if $\ap<\al$ and $\al_{|\al|+1}\in \varepsilon(B)\backslash F_B$ then $D_{(\ap,B)F_B}\cap D_{(\al,A)F_A} = D_{(\al,A)F_A}$ and empty otherwise.  

(iv) In any other case  $D_{(\al,A)F_A}\cap D_{(\ap,B)F_B}= \emptyset$.
\end{proof}

\begin{lema}\label{lem:disjointunionofcylindersets}
 Let $U\subseteq X$ be a compact open set. Then $U$ can be written as a disjoint union of basic open cylinders sets. 
\end{lema}
\begin{proof}
Since $U$ is open and the topology on $X$ is second-countable, we can express $U$ as a countable union of basic open cylinder sets. Since $U$ is compact, this union may be taken to be finite, say 
\begin{equation}\label{eq:finiteunion}
     U= \bigcup_{i=1}^{N} D_{(\al_i,A_i)F_i}.
\end{equation}
If $D_{(\al_i,A_i)F_i}\cap D_{(\al_j,A_j)F_j}\neq\emptyset$ for some $i\neq j$, then $\al_i\leq\al_j$ or $\al_j\leq\al_i$, by Lemma \ref{lem:cilynderIntersections}. We may assume without loss of generality that $\al_i\leq\al_j$. If $\al_i<\al_j$, then $D_{(\al_j,A_j)F_j} \subseteq D_{(\al_i,A_i)F_i}$, and we may omit $D_{(\al_j,A_j)F_j}$ from the union in Equation (\ref{eq:finiteunion}). If $\al_i=\al_j$, then $D_{(\al,A)F_A}\cup D_{(\al,B)F_B}=D_{(\al,A\cup B)(F_A\cap F_B)}$, which is again a basic open set. By either taking the bigger set or by taking the union for all pairs of basic open sets  in Equation (\ref{eq:finiteunion}) with non-empty intersection we end up with a disjoint union. 
\end{proof}

\begin{lema}\label{lem:basicneighbintersections}
Let $\al,\be,\ap,\bp\in\GG^*$ and $A,B\in\GG^0$ be such that \newline $(\al,A), (\be,A), (\ap,B), (\bp,B) \in \mathfrak{p}$, and let $F_A\subseteq \varepsilon(A)$ and $F_B\subseteq \varepsilon(B)$ be finite subsets. Then 
$$ \mathcal{Z}(\alpha,\beta,A,F_A)\cap \mathcal{Z}(\ap,\bp,B,F_B) $$
equals either 
\begin{enumerate}
    \item $\emptyset$, or  
    \item $\mathcal{Z}(\alpha,\beta,A,F_A)$, or  
    \item $\mathcal{Z}(\ap,\bp,B,F_B) $, or 
    \item\label{item:lemma.intersection} $\mathcal{Z}(\al,\be,A\cap B, F_A\cup F_B)$.
    
    In addition, if $\al=\ap,\be=\bp$, then we also have that $$\mathcal{Z}(\alpha,\beta,A,F_A)\cup \mathcal{Z}(\ap,\bp,B,F_B)=\mathcal{Z}(\al,\be,A\cup B, F_A\cap F_B).$$
\end{enumerate}
\end{lema}
\begin{proof}
Note that if $\mathcal{Z}(\alpha,\beta,A,F_A)\cap \mathcal{Z}(\ap,\bp,B,F_B)\neq \emptyset$, then $D_{(\al,A)F_A}\cap D_{(\ap,B)F_B}\neq \emptyset$,  $D_{(\be,A)F_A}\cap D_{(\bp,B)F_B}\neq \emptyset$ and $|\al|-|\be|=|\ap|-|\bp|$. 

By Lemma \ref{lem:cilynderIntersections} we may assume without loss of generality that $\al\leq \ap$. The same applies to $D_{(\be,A)F_A}\cap D_{(\bp,B)F_B}$. We claim that $\al=\ap$ if and only if $\be=\bp$, and also that $\al<\ap$ if and only if $\be<\bp$. To see this, first assume that $\al=\ap$. Then $|\al|=|\ap|$, which together with $|\al|-|\be|=|\ap|-|\bp|$ imply that 
$|\be|=|\bp|$. Let $x\in D_{(\be,A)F_A}\cap D_{(\bp,B)F_B}$. Then $x=\be\gamma=\bp\delta$, for some $\gamma,\delta\in X$, which implies that $\be\leq \bp$ or $\bp\leq\be$. Since  $|\be|=|\bp|$, it follows that $\be=\bp$. A similar argument gives the converse.

Assume that $\al<\ap$. Then $|\al|\leq|\ap|$ and $|\al|-|\be|=|\ap|-|\bp|$ imply that $|\be|\leq|\bp|$. Similarly to above $D_{(\be,A)F_A}\cap D_{(\bp,B)F_B}\neq\emptyset$ implies that $\be\leq \bp$ or $\bp\leq\be$. However, if $|\be|\leq|\bp|$ then $\be\leq\bp$. The converse follows from a similar argument. 

Next we  show that  $ \mathcal{Z}(\alpha,\beta,A,F_A)\cap \mathcal{Z}(\ap,\bp,B,F_B) $ is equal to one of the sets in (1)-(4). If $\ap<\al$ and (necessarily) $\bp<\be$, then by Lemma \ref{lem:cilynderIntersections} we have $D_{(\al,A)F_{A}}\cap D_{(\ap,B)F_B}=D_{(\al,A)F_A}$ and $D_{(\be,A)F_A}\cap D_{(\bp,B)F_B}=D_{(\be,A)F_A}$, which implies that $\mathcal{Z}(\alpha,\beta,A,F_A)\subset  \mathcal{Z}(\ap,\bp,B,F_B)$ and gives (2). Similarly, if $\al<\ap$ and (necessarily) $\be<\bp$, then Lemma \ref{lem:cilynderIntersections} implies that $ \mathcal{Z}(\ap,\bp,B,F_B)\subset \mathcal{Z}(\alpha,\beta,A,F_A)$, which gives (3). Finally, if $\al=\ap$ and $\be=\bp$, then applying Lemma \ref{lem:cilynderIntersections} again gives (4). The union in (4) is clear from Lemma \ref{lem:cilynderIntersections}. In all other cases we have the empty set, which completes the proof.
\end{proof}

\begin{lema}\label{lem:thirdbasisfortop}
The collection 
$$ \mathcal{T}=\{\mathcal{Z}(\alpha,\beta,A,F_A)\mid (\alpha,A),(\beta,A)\in\mathfrak{p}, F_A\subseteq \varepsilon (A) \text{ is finite} \} $$
forms a basis for the topology on $\Gs$. Moreover, every  compact open set in $\Gs$ can be written as a disjoint union consisting of sets from $\mathcal{T}$. 
\end{lema}
\begin{proof}
We first show that $\mathcal{T}$ is basis by showing every element in $\mathscr{C}$  can be written as a union of elements from  $\mathcal{T}$. 

 Fix  $\mathcal{Z}(U,\alpha,A,\beta,V)\in\mathscr{C}$. Then $\sigma^{|\alpha|}(U)=\sigma^{|\beta|}(V)$, with  $U\subseteq D_{(\alpha,A)}$ and $V\subseteq D_{(\beta,A)}$. Let $W=\sigma^{|\al|}(U)=\sigma^{|\be|}(V)$. By \cite[Prop 3.16]{MR3938320} the shift map is a homeomorphism on basic open sets of $X$. Hence $W$ is compact and open, so we can write 
$$ W=\bigcup_{i=1}^{N}D_{(\gamma_i,B_i)F_{B_i}}.  $$
Since $\sigma^{|\al|}$ and $\sigma^{|\be|}$ are injective on $U$ and $V$, it follows that
\begin{equation}
\begin{split}
    & U=\bigcup_{i=1}^{N}D_{(\al\gamma_i,B_i)F_{B_i}} \hspace{0.4cm} (\text{with } s(\gamma_1)\in A), \text{ and }  \\
    & V=\bigcup_{i=1}^{N}D_{(\be\gamma_i,B_i)F_{B_i}}, \\
\end{split}
\end{equation}
for some $\gamma\in \GG^*$. Since $\sigma^{|\al|}(D_{(\al\gamma_i,B_i)F_{B_i}})=\sigma^{|\be|}(D_{(\al\gamma_j,B_j)F_{B_j}})$ if and only if $i=j$, and $\sigma^{|\al|}(D_{(\al\gamma_i,B_i)F_{B_i}})\cap \sigma^{|\be|}(D_{(\al\gamma_j,B_j)F_{B_j}})=\emptyset$ otherwise, it follows that 
$$\mathcal{Z}(U,\alpha,A,\beta,V)=\bigcup_{i=1}^{N} \mathcal{Z}(\alpha\gamma_i,\beta\gamma_i,B_i,F_{B_i})   .$$

Now, since $\mathcal{T}$ is a basis for the topology on $\Gs$, any compact open set can be written as a union of elements from $\mathcal{T}$, which in turn can be written as a disjoint union by applying  Lemma \ref{lem:basicneighbintersections} to non-empty intersections.
\end{proof}

The following lemma describes the bisections in $\Gs$.

\begin{lema}\label{lem:bisectioncharacterize}
Let $U\subset \Gs$ be a compact open bisection  with $s(U)=r(U)$. Then 
$$ U=\bigsqcup_{i=1}^{N} \mathcal{Z}(\al_i,\be_i,A_i,F_{A_i}), $$
where,  for $i=1,\ldots,N$, we have $A_i\in\GG^0$, $(\al_i,A_i),(\be_i,A_i)\in\mathfrak{p}$ and  $F_{A_i}\subset\varepsilon(A_i)$ finite. Moreover,
$$  s(U)=\bigsqcup_{i=1}^{N} D_{(\be_i,A_i)F_{A_i}}=\bigsqcup_{i=1}^{N} D_{(\al_i,A_i)F_{A_i}}=r(U).$$
 
\end{lema}

\begin{proof}
Since $U$ is a compact open bisection, we can write 
$$ U=\bigsqcup_{i=1}^{N} \mathcal{Z}(\al_i,\be_i,A_i,F_{A_i}), $$
by Lemma \ref{lem:thirdbasisfortop}. That $s(U)=\sqcup_{i=1}^{N} D_{(\be_i,A_i)F_{A_i}} $ and $r(U)=\sqcup_{i=1}^{N} D_{(\al_i,A_i)F_{A_i}}$ follow from the definition of $\mathcal{Z}(\al_i,\be_i,A_i,F_{A_i})$ and the fact that the range and source maps restricted to $U$ are injective, and hence preserve disjoint unions.
\end{proof}

We now characterize  elements of the topological full group of the groupoid $\Gs$ associated with an ultragraph that satisfies Condition (L) (that is, $\Gs$ is effective by Propostion \ref{prop:effectiveCharac}).

\begin{proposicao}\label{prop:characbisecandfullgrp}
Let $\GG$ be an ultragraph with no sinks and that satisfies Conditions (RFUM) and Condition (L). If $\pi_U\in\llbracket \Gs\rrbracket$, then the full bisection $U\subseteq \Gs$ can be written as $$ U=\left( \bigsqcup_{i=1}^{N} \mathcal{Z}(\al_i,\be_i,A_i,F_{A_i}) \right) \bigsqcup \left(X\backslash \supp(\pi_U)  \right), $$ 
where,  for $i=1,\ldots,N$, we have $A_i\in\GG^0$, $(\al_i,A_i),(\be_i,A_i)\in\mathfrak{p}$ and  $F_{A_i}\subset\varepsilon(A_i)$ finite, and 
$$\supp(\pi_U)=\bigsqcup_{i=1}^{N} D_{(\be_i,A_i)F_{A_i}}=\bigsqcup_{i=1}^{N} D_{(\al_i,A_i)F_{A_i}}  .$$
The paths $\al_1,\dots,\al_N$ are pairwise disjoint, as are the paths $\be_1,\dots,\be_N$, and $\al_i$ and $\be_i$ are distinct for each $0\leq i\leq N$. The homeomorphism $\pi_U:s(U)\to r(U)$ is given by $\pi_U(\be_i x)=\al_i x$ for $\be_i x\in D_{(\be_i,A_i)F_{A_i}}$ and the identity otherwise.
\end{proposicao}
\begin{proof}
The proof follows directly from Lemma \ref{lem:bisectioncharacterize} and Lemma \ref{bisectiondecomposition}.
\end{proof}

\section{Equivalence of groupoid and topological full group isomorphisms} \label{sec:MainResults}
 In this section we prove our main results, Theorem \ref{thm:MainIsoUGs} and Theorem \ref{main2}, which give conditions under which ultragraph groupoids are isomorphic if and only if their topological full groups are (algebraically) isomorphic. These generalize \cite[Theorem 10.10 and Theorem 10.11]{MR3950815} from graphs to the ultragraphs. The techniques and ideas in this section are based on \cite{MR3950815}, but adapted to ultragraphs. Throughout we highlight some of the subtle differences between the graph and ultragraph case.

We begin by defining three conditions for ultragraphs that generalize that of \cite[Definition 10.1]{MR3950815}. Fix an ultragraph $\GG$ with no sinks and that satisfies Condition (RFUM). If $A,B\in\GG^0$, then we let
$$ A\mathfrak{p} B:=\{(\al,C)\in\mathfrak{p}:|\al|\geq 1, s(\al,C)\in A, B\subseteq C\}.  $$ 
\begin{itemize}
    \item $\GG$ satisfies Condition (K) if for every $v\in G^0$, there is either no simple loop based at $v$ or at least two simple loops based at $v$.
    \item $\GG$ satisfies Condition (W) if for every wandering path $\al=\al_1\ldots\in\mathfrak{p}^{\infty}$,  we have that, for some $i\in\N$,
    \begin{equation}\label{conditionW}
        |\{(\be,C)\in s(\al)\mathfrak{p} s(\al_{i+1})\mid \be\neq\al_1\ldots\al_i\}|\geq 1.
    \end{equation}     
    \item $\GG$ satisfies Condition ($\infty$) if for every minimal infinite emitter $A\in\GG^0$ we have that 
    $$|\{e\in \varepsilon(A):r(e) \mathfrak{p} A\neq \emptyset\}|=\infty .$$
\end{itemize}

\begin{remark}
We remark on some subtleties in conditions (K),(W) and ($\infty$) for ultragraphs when compared with graphs. Firstly, in contrast to graphs, a simple loop $\al_1,\ldots,\al_n$ in an ultragraph may have $r(\al_i)\cap r(\al_j)\neq\emptyset$. Secondly, Condition (W) is intended to provide us with an arbitrary number of disjoint paths on a wandering path with no cycles. For this the condition that $\be\neq\al_1\ldots\al_i$ in (\ref{conditionW}) is crucial, because edge ranges can be sets. For example,  if $r(\al_n)=\{s(\al_{n+1}),v_1,v_2\}$ then $(\al_1\ldots\al_n,\{s(\al_{n+1}),v_1\})$ and $(\al_1\ldots\al_n,\{s(\al_{n+1}),v_2\})$ are two distinct ultrapaths in $s(\al)\mathfrak{p} s(\al_{i+1})$, but not disjoint as required later on. So, unlike graphs the cardinality of the set in (\ref{conditionW}) may be be infinite, but with only one path (as opposed to an ultrapath).  
\end{remark}


Conditions (K) and (W) are intended to provide us with an arbitrary number of disjoint path for a given path. This will be crucial in Theorem \ref{thm:main1KWInf}. In the following lemmas we illustrate this fact by considering loops and wandering paths separately, and then combining these results in Lemma~\ref{lem:infdisjointpaths}.

If there are two distinct simple loops based at a vertex in a graph, then these loops are necessarily disjoint. This is not the case for ultragraphs. For example, if $G^0=\{v_1,v_2,v_3\}$, $\GG^1=\{e_1,e_2\}$ with $s(e_i)=v_i$, $r(e_1)=\{v_1, v_2\}$ and $r(e_2)=\{v_1,v_3\}$. Then $e_1$ and $e_1e_2$ are two distinct simple loops based at $v_1$, but they are not disjoint. However, we can still find infinitely many disjoint loops as the following lemma illustrates.

\begin{lema}\label{lem:CondkInfLoops}
Suppose $\GG$ satisfies Condition (K). If $\rho$ is a loop, then there are infinitely many loops $\al_1,\ldots$ based at $s(\rho)$ such that $\rho,\al_1,\ldots$ are all pairwise disjoint and $r(\rho)\cap r(\al_i)\neq \emptyset$ for $i=1,2,\ldots$.
\end{lema}
\begin{proof}
Since $\rho$ is a loop Condition (K) implies that there are at least two distinct simple loops $\tau_1$ and $\tau_2$ based at $s(\rho)$ (but possibly not disjoint). 

Suppose $\tau_2$ has minimal length between all simple loops based on $s(\rho)$. If $\tau_1$ and $\tau_2$ are disjoint, then $\tau_1, \tau_2^{}\tau_1,\tau_2^{2}\tau_1, \ldots$ are disjoint and $s(\tau_1)=s(\tau_2)\in r(\tau_1)=r(\tau_2^{n}\tau_1)$ for every $n\in\N$. 

Otherwise, $\tau_1=\tau_2\beta$ for some $\beta\in\GG^*$. We claim that $\tau_2\tau_2\beta$ and $\tau_2\beta$ are disjoint.  If $\tau_2\beta$ is a initial segment of $\tau_2\tau_2\beta$, then $\beta$ is a initial segment of $\tau_2\beta$. In this case $s(\beta)=s(\tau_2)=s(\rho)$, and since $r(\beta)=r(\tau_1)$, $\beta$ would be a loop based on $s(\rho)$. Since $\tau_2$ has minimal length, we have that $|\tau_2|\leq|\beta|$, which would imply that $\beta=\tau_2\alpha$  for some $\alpha\in\GG^*$ and that $\tau_1=\tau_2\tau_2\beta$, which contradicts the fact that $\tau_1$ is a simple loop. Now, repeating the argument above for  $\tau_2\tau_2\alpha$ and $\tau_2\beta$ gives the required result.
\end{proof}

In the graph case, if $\alpha$, $\beta$ and $\gamma$ are paths such that $\alpha=\beta\gamma$ and $r(\beta)=r(\alpha)=s(\gamma)$, then $\gamma$ is necessarily a loop. For ultragraphs, if we replace the equality with $s(\gamma)\in r(\al)\cap r(\be)$, then $\gamma$ is not necessarily a loop. For example, say $r(\alpha)=\{v,w\}$ and $\gamma$ is a path of length one such that $s(\gamma)=v$ and $r(\gamma)=\{w\}$. The proof of the following lemma has to take this into account and is different from \cite[Lemma~10.2(2)]{MR3950815}.

\begin{lema}\label{lem:Winfdisjointpaths}
Let $\GG$ be an ultragraph that satisfies condition (W) and $\al$ a wandering path such that it has no loop based at $s(\alpha_i)$ for any $i$. Then for any given $N\in\N$ there exists $n\in\N$ and $N+1$ disjoint ultrapaths in $s(\al)\mathfrak{p}s(\al_{n+1})$, one of which is an initial segment of $\al$.
\end{lema}

\begin{proof}
    We prove by induction on $N$. If $N=0$, we can choose $n=1$, since $(\alpha_1,r(\alpha_1))\in s(\al)\mathfrak{p}s(\al_2)$. Fix $N\in\N$, and suppose that $n\in\N$ is such that there exist $N+1$ disjoint paths $(\beta^{(1)},B_1),\ldots,(\beta^{(N+1)},B_{N+1})\in s(\al)\mathfrak{p}s(\al_{n+1})$. We can assume without loss of generality that $B_1=\cdots=B_{N+1}=\{s(\alpha_{n+1})\}$ and that $s(\al_m)\neq s(\al_{n+1})$ for all $m>n+1$, the latter because the path is wandering. Since $\alpha_{n+1}\alpha_{n+2}\ldots$ is also wandering, by condition (W), there exists $m>n+1$ such that $|\{(\gamma,A)\in s(\al_{n+1})\mathfrak{p} s(\al_{m+1})\mid \gamma \neq\al_{n+1}\ldots\al_m\}|\geq 1$. Let $\gamma\neq \alpha_{n+1}\ldots\alpha_m$ be such that $(\gamma,A)\in s(\al_{n+1})\mathfrak{p} s(\al_{m+1})$ for some $A\in\GG^{0}$. Again, we may assume without loss of generality that $A=\{s(\al_{m+1})\}$. We claim that the ultrapaths $(\beta^{(1)}\alpha_{n+1}\ldots\alpha_m,A),\ldots,(\beta^{(N+1)}\alpha_{n+1}\ldots\alpha_m,A),(\alpha_1\ldots\alpha_n\gamma,A)$ are mutually disjoint in $s(\al)\mathfrak{p}s(\al_{m+1})$. We start by noticing, that for $i,j\in\{1,\ldots,N+1\}$ with $i\neq j$,  $(\beta^{(i)}\alpha_{n+1}\ldots\alpha_m,A)$ and $(\beta^{(j)}\alpha_{n+1}\ldots\alpha_m,A)$ are still disjoint. 
    
    We now prove that $(\beta^{(i)}\alpha_{n+1}\ldots\alpha_m,A)$ and $(\alpha_1\ldots\alpha_n\gamma,A)$ are disjoint for any $i=1, \ldots, N+1$. We divide in six cases:
    \begin{itemize}
        \item If $\beta^{(i)}$ is disjoint from $\alpha_1,\ldots,\alpha_n$, then $\beta^{(i)}\alpha_{n+1}\ldots\alpha_m$ is also disjoint from $\alpha_1\ldots\alpha_n\gamma$.
        \item If $\beta^{(i)}=\alpha_1\ldots\alpha_k$ for some $k<n$, then $s(\alpha_{k+1})\neq s(\alpha_{n+1})$, otherwise $\alpha_{k+1}\ldots\alpha_n$ is a loop based on $s(\alpha_{k+1})$. In this case $\beta^{(i)}\alpha_{n+1}\ldots\alpha_m$ is disjoint from $\alpha_1\ldots\alpha_n\gamma$ because the $(k+1)$-coordinates are different.
        \item If $\beta^{(i)}=\alpha_1\ldots\alpha_n\beta'$ for some $\beta'$ with $|\beta'|>0$, then $s(\beta')\neq s(\alpha_{n+1})$, otherwise $\beta'$ is a loop based on $s(\alpha_{n+1})$. In this case $\beta^{(i)}\alpha_{n+1}\ldots\alpha_m$ is disjoint from $\alpha_1\ldots\alpha_n\gamma$ because $s(\gamma)=s(\alpha_{n+1})$ so that $\beta'_1\neq\gamma_1$.
        \item If $\beta^{(i)}=\alpha_1\ldots\alpha_n$ and $\alpha_{n+1}\ldots\alpha_n$ and $\gamma$ are disjoint, then $\beta^{(i)}\alpha_{n+1}\ldots\alpha_m$ is also disjoint from $\alpha_1\ldots\alpha_n\gamma$.
        \item Suppose that $\beta^{(i)}=\alpha_1\ldots\alpha_n$ and $\gamma=\alpha_{(n+1)}\ldots\alpha_{m}\gamma'$ for some $\gamma'$. Notice that $s(\gamma')\neq s(\alpha_{m+1})$, because otherwise $\gamma'$ is a loop based at $s(\alpha_{m+1})$ since $s(\alpha_{m+1})\in r(\gamma)=r(\gamma')$. This implies that $(\beta^{(i)}\alpha_{n+1}\ldots\alpha_m,A)$ is disjoint from $(\alpha_1\ldots\alpha_n\gamma,A)$.
        \item Suppose that $\beta^{(i)}=\alpha_1\ldots\alpha_n$ and $\gamma=\alpha_{n+1}\ldots\alpha_{k}$ for some $n+1\leq k < m$. As in the last case, we must have $s(\alpha_{k+1})\neq s(\alpha_{m+1})$, otherwise $\alpha_{k+1}\ldots\alpha_m$ is a loop based on $s(\alpha_{k+1})$. Again, this implies that $(\beta^{(i)}\alpha_{n+1}\ldots\alpha_m,A)$ is disjoint from $(\alpha_1\ldots\alpha_n\gamma,A)$.
    \end{itemize}
    Hence there are $N+2$ mutually disjoint ultrapaths in $s(\al)\mathfrak{p}s(\al_{m+1})$.
    
    The last part follows from the induction process.
\end{proof}

 The following lemma provide us with four disjoint paths when Conditions (K) and (W) are satisfied (the proof could also be easily modified to provide us with any desired number of disjoint paths, but we will only need four).

\begin{lema} \label{lem:infdisjointpaths}
    Suppose that $\GG$ satisfies conditions (K) and (W). Then for any $x\in\mathfrak{p}^{\infty}$ and any $n\in\N$, there exist $m>n$ and four paths $\alpha_1=x_{n+1}\ldots x_m$, $\alpha_2$, $\alpha_3$ and $\alpha_4$ such that $s(x_{m+1})\in r(\alpha_i)$ and $s(\al_i)=s(x_{n+1})$ for all $i$, and the ultrapaths $(x_1\ldots x_n\alpha_i,\{s(x_{m+1})\})$, $i=1,\ldots,4$ are mutually disjoint.
\end{lema}

\begin{proof}
    We divide in three cases. First, suppose that for some $n<p<m$, $x_p\ldots x_m$ is a loop. Then applying Lemma \ref{lem:CondkInfLoops}, we find the paths $\alpha_i$.
    
    Now suppose, that there is no segment after $x_n$ of $x$ that is a loop, but there exists $p>n$ and a loop $\rho$ based on $s(x_{p+1})$. We find three disjoint loops $\alpha'_2,\alpha'_3,\alpha'_4$ based on $s(x_{p+1})$. By choosing $m>p+\max_{i=2,3,4}{|\alpha'_i|}$, we have that $x_{p+1}\ldots x_m$ is disjoint from each $\alpha'_i$, because of the choice of $m$ and the fact that no segment of $x$ after $x_n$ is a loop. In this case $\alpha_1=x_{n+1}\ldots x_m$ and $\alpha_i=x_{n+1}\ldots x_p\alpha'_ix_{p+1}\ldots x_m$ for $i=2,3,4$ are mutually disjoint.
    
    Finally, if there is no loop based at any $s(x_i)$ for $i>m$, then we can apply Lemma \ref{lem:Winfdisjointpaths} to find the appropriate paths.
\end{proof}

Suppose that $\GG$  satisfies conditions (K),(W) and ($\infty$).
We want to employ \cite[Theorem 6.2]{MR3950815} to show that we have an equivalence between isomorphisms of ultragraph groupoids and isomorphisms of their topological full groups. For this we need to show that the class of all pairs $(\Gamma,X)$, where $\Gamma$ is a subgroup of $\llbracket \Gs\rrbracket$ containing the commutator subgroup and $X$ is the edge shift space, is a faithful class of space-group pairs (see Definition \ref{dfn:SpaceGroupPairsFaithful}). By \cite[Theorem 6.6]{MR3950815} for such a class to be faithful it is sufficient to show that the following properties are satisfied (see also \cite[Definition 6.3]{MR3950815}): 
\begin{enumerate}
    \item[(F1)] For $x\in X$ and any clopen neighborhood $A\subset X$ of  $x$ there exists an involution $\phi\in \Gamma$ such that $x\in\supp(\phi)$ and $\supp(\phi)\subseteq A$. 
    \item[(F2)] For any involution  $\phi\in \Gamma\backslash \{1\}$, and any non-empty clopen set $A\subseteq \supp(\phi)$, there exists a $\psi\in \Gamma\backslash \{1\}$ such that $\supp(\psi)\subseteq A\cup \phi(A)$ and $\phi(x)=\psi(x)$ for every $x\in \supp(\psi)$.
    \item[(F3)] For any non-empty clopen set $A\subset X$, there exists $\phi\in\Gamma$ such that $\supp(\phi)\subseteq A$ and $\phi^2\neq 1$.
\end{enumerate}
Let $K^F$ denote the class of all space-group pairs that satisfy Conditions F1, F2 and F3 above.

The following theorem extends \cite[Theorem 10.3]{MR3950815} to ultragraphs. The structure of the the proof is essentially the same as that of \cite[Theorem 10.3]{MR3950815} (which is based on Matui's proof \cite[Proposition 3.6]{MR3377390}), since we now have the description of the full bisections and topological full group of $\Gs$ in an ultragraph context (Proposition \ref{prop:characbisecandfullgrp}).  

\begin{teorema} \label{thm:main1KWInf}
Let $\GG$ be an ultragraph with no sinks and that satisfies Condition (RFUM). Let $\Gamma$ be a subgroup of $\llbracket \Gs\rrbracket$ containing the commutator subgroup $D(\llbracket \Gs\rrbracket)$. Then $(\Gamma, \Gs^{(0)})\in K^F$  if and only if $\GG$  satisfies Conditions (K), (W) and ($\infty$).
\end{teorema}
\begin{proof}
We first show that Conditions (K) and (W) imply Conditions (F2) and (F3), and then that (K), (W) and ($\infty$) are necessary and sufficient for (F1) to hold. 

Assume Conditions (K) and (W) are satisfied. We first show (F3). Let $A\subset X$ be a non-empty clopen set. Then there is $x\in A$ and a cylinder set $D_{(\be,B)F}\subseteq A$ containing $x$, for some $(\be,B)\in\mathfrak{p}$ and finite set $F\subset \varepsilon(B)$. Since $\GG$ does not have sinks, there is an edge $e\in \GG^1$ with $s(e)\in B$. Then $D_{(\be e,r(e))}\subseteq D_{(\be,B)F}$. So we may assume without loss of generality that for any non-empty clopen set $A\subset X$ there is an ultrapath $(\be,B)\in\mathfrak{p}$ such that $D_{(\be,B)}\subseteq A$ (without a set $F\subset \varepsilon(B)$). Since $\GG$ has no sinks, we have that there exists at least one infinite path $x\in\mathfrak{p}^{\infty}$ such that $s(x)\in B$. By Lemma~\ref{lem:infdisjointpaths}, there are three disjoint paths $\al_1,\al_2,\al_3$ based at $B$ such that $C=r(\al_1)\cap r(\al_2)\cap r(\al_3)\neq \emptyset$.  Let  $V=\mathcal{Z}(\be\al_1,\be\al_2,C, \emptyset)$ and $W=\mathcal{Z}(\be\al_2,\be\al_3,C,\emptyset)$, and define $\pi_{\hat{V}}$ and $\pi_{\hat{W}}$ as in Lemma~\ref{NonEmptyIntersecDfnsInvolution}. Now define $\La=[\pi_{\hat{V}},\pi_{\hat{W}}]$.
Let $y\in X$ be such that $s(y)\in C$. Then  $\La(\be\al_3y)=\be\al_2y, \La(\be\al_2y)=\be\al_1y$ and $\La(\be\al_1y)=\be\al_3y$ (under the convention that $[g,h]=g^{-1}h^{-1}gh)$. Hence $\La^2\neq\mathrm{id}$ and $\La^3=\mathrm{id}$. Also, since $D_{(\be\al_i,C)}\subseteq D_{(\be,B)}\subseteq A$ for $i=1,2,3$, it follows that $\supp(\La)\subseteq A$. Hence (F3) is satisfied. 

Next we show that (F2) is satisfied. Let $\tau\in \Gamma\backslash \{1\}$ and let $A\subseteq \supp(\tau)$ be a non-empty clopen set. It follows from Proposition \ref{prop:characbisecandfullgrp} that $\tau=\pi_U$ where $U$ is a full bisection in $\Gs$ and can be written as 
 $$ U=\left( \bigsqcup_{i=1}^{N} \mathcal{Z}(\al_i,\be_i,A_i,F_{A_i}) \right) \bigsqcup \left(X\backslash \supp(\pi_U)  \right). $$
Similarly to the first part of the proof we can find an ultrapath $(\gamma,B)\in\mathfrak{p}$ such that $D_{(\gamma,B)}\subseteq A\cap D_{(\be_j,A_j)F_{A_j}} $ for some index $1\leq j\leq N$, and two disjoint paths $\la_1$ and $\la_2$ based at  $B$ such that $C=r(\la_1)\cap r(\la_2)\neq \emptyset$.  We may assume without loss of generality that $|\gamma|>|\beta_j|$, so that  $\gamma=\be_j\rho$ for some $\rho\in \GG^*$ with $|\rho|\geq 1$ and $\rho_1\notin F_{A_j}$. Define the following bisections 
\begin{equation*}
\begin{split}
    & V = \mathcal{Z}(\be_j\rho\la_1,\be_j\rho\la_2,C,\emptyset)\sqcup \mathcal{Z} (\al_j\rho\la_1,\al_j\rho\la_2,C,\emptyset) \\
    & W = \mathcal{Z}(\al_j\rho\la_1,\be_j\rho\la_2,C,\emptyset).
\end{split}
\end{equation*}
Since $\tau=\pi_U$ is an involution we have that $\tau(\be_j x)=\al_j x$, for $\beta_j x\in D_{(\be_j,A_j)F_{A_j}} $, and  $\tau(\al_j x)=\be_j x$, for $\alpha_j x\in D_{(\al_j,A_j)F_{A_j}} $. Define $\La=[\pi_{\hat{V}},\pi_{\hat{W}}]$. Then $\La\in \Gamma$, and 
\begin{equation*}
\begin{split}
    \supp(\La)& = D_{(\be_j\rho\la_1)}\sqcup D_{(\be_j\rho\la_2)}\sqcup D_{(\al_j\rho\la_1)}\sqcup D_{(\al_j\rho\la_2)} \\
    &\subseteq D_{(\gamma,B)} \cup \tau(D_{(\gamma,B)})  \\
    &\subseteq A\cup \tau(A).
\end{split}
\end{equation*}
Since both $\tau$ and $\La$ interchange the initial paths $\al_j$ and $\be_j$, it follows that they agree on $\supp(\La)$. Hence (F2) is satisfied.

Next we show that (F1) holds if and only if Conditions (K), (W) and ($\infty$) hold.
Since properties (F1) and (F3) fail in the presence of isolated points, we assume for the remainder of the proof that that $\GG$ has no sinks, satisfies (RFUM), has no semi-tails and satisfies Condition~(L) (see Proposition \ref{prop:isoltedpoints}), and we fix an $x\in X$ and a clopen neighborhood $A\subset X$ of $x$. 

Assume that Conditions (K), (W) and ($\infty$) hold. First assume that $x\in\mathfrak{p}^{\infty}$. By choosing $m$ big enough we may assume that $D_{(x_1\ldots x_m,r(x_m))}\subset A$. By Lemma~\ref{lem:infdisjointpaths}  
we can find three mutually disjoint paths $\be^{(1)},\be^{(2)},\be^{(3)}\in s(x_{m+1})\mathfrak{p}s(x_{n+1})$, all of which are also disjoint from $x_{m+1}\ldots x_n$. Now, we put $\al^{(1)}=x_1\ldots x_m\be^{(1)}$, $\al^{(2)}=x_1\ldots x_m\be^{(2)}$, $\al^{(3)}=x_1\ldots x_m\be^{(3)}$ and $\al^{(4)}=x_1\ldots x_n$. Note that $B=r(\al^{(1)})\cap r(\al^{(2)}) \cap r(\al^{(3)}) \cap r(\al^{(4)})\neq \emptyset$. Define bisections 
\begin{equation*}
\begin{split}
    V&=\mathcal{Z}(\al^{(1)},\al^{(2)},B,\emptyset)\sqcup \mathcal{Z}(\al^{(3)},\al^{(4)},B,\emptyset), \text{ and } \\
    W&= \mathcal{Z}(\al^{(1)},\al^{(3)},B,\emptyset).
\end{split}
\end{equation*}
Then $\pi=[\pi_{\hat{V}},\pi_{\hat{W}}]\in \Gamma$ with $\supp(\pi)=\sqcup_{i=1}^{4}D_{(\al^{(i)},B)}\subseteq D_{(x_1\ldots x_m,r(x_m))}\subset A $, $\pi^2=\mathrm{id}$ and $x \in D_{(\al^{(4)},B)}\subseteq \supp(\pi)$. Hence (F1) is satisfied. 

Secondly, assume that $x=(\be,B)\in X_{fin}$. Since $B$ is a minimal infinite emitter, there is a finite set $F\subset \varepsilon(B)$ such that $D_{(\be,B)F}\subseteq A$.  Condition ($\infty$) implies that  $|\{e\in \varepsilon(B):r(e) \mathfrak{p} B\neq \emptyset\}|=\infty$. Hence we can find three disjoint loops $\al^{(1)},\al^{(2)},\al^{(3)}$ based at $B$ such that the edges $\al^{(i)}_1\notin F,i=1,2,3$. Put $H=F\cup\{\al^{(1)}_1,\al^{(2)}_1,\al^{(3)}_1\}$ and define 
\begin{equation*}
\begin{split}
    V&=\mathcal{Z}(\be\al^{(1)},\be,B,H)\sqcup \mathcal{Z}(\be\al^{(2)},\be\al^{(3)},B,H), \text{ and } \\
    W&= \mathcal{Z}(\be\al^{(1)},\be\al^{(2)},B,H).
\end{split}
\end{equation*}
Then $\pi=[\pi_{\hat{V}},\pi_{\hat{W}}]\in \Gamma$ with $\supp(\pi)=\sqcup_{i=1}^{3}D_{(\be\al^{(i)},B)H}\sqcup D_{(\be,B)H}\subseteq D_{(\be,B)F}\subset A $, $\pi^2=\mathrm{id}$ and $x \in D_{(\be,B)H}\subseteq \supp(\pi)$. Hence (F1) is satisfied.

Next assume properties (F1)-(F3) hold. We show this implies conditions (K),(W) and $(\infty)$. We do this separately for each. We begin with Condition (K). Assume that $\GG$ dos not satisfy condition (K). Then there is a vertex $v\in G^0$ with a single simple loop $\gamma$ based at $v$. Since $\GG$ satisfies condition (L) the loop $\gamma$ has an exit. However, since we assume that there are no sinks, there exits an edge $e\in \GG^1$ such that $s(e)=s(\gamma_i)$ and $e\neq \gamma_i$, for some $i\in\N$. Let $x=\gamma^{\infty}$ and $A=D_{(\gamma,r(\gamma))}$. We show that (F1) does not hold for this pair. Assume there is a $\pi_U\in \llbracket \Gs\rrbracket$ such that $\gamma^{\infty}\in \supp(\pi_U)\subseteq D_{(\gamma,r(\gamma))}$. Then there is a basic open neighborhood $\mathcal{Z}(\al,\be,B,\emptyset)\subseteq U$ (Proposition \ref{prop:characbisecandfullgrp}) with $(\al,B),(\be,B)\in\mathfrak{p}$, $\al\neq\be$ and $\gamma^{\infty}\in D_{(\be,B)}\subset D_{(\gamma,r(\gamma))}$. Hence $\be=\gamma^m\rho$, for some $m\in\N$ and $\rho\in\GG^*$ with $|\rho|<|\be|$. Therefore, by extending $\al$ and $\be$ if necessary we may assume that $\be=\gamma^m$.  Similarly, $D_{(\al,B)}\subset D_{(\gamma,r(\gamma))}$ and we may assume that $\al=\gamma^n$. However, since $\al\neq \be$ and $\gamma$ is the only simple loop based at $v$, it follows that $m\neq n$. Let $z\in \varepsilon(r(e))$. Then $(\pi_U)^2(\gamma^{2m}ez)=\gamma^{2n}ez\neq \gamma^{2m}ez$. Hence $\pi_U$ is not an involution, which implies that (F1) is not satisfied. 

We show next the necessity off Condition~(W) for (F1) to be satisfied. Assume that Condition (W) does not hold. Then there is an infinite wandering path $x=x_1x_2\cdots $ such that 
\begin{equation}\label{eqn:NotConditionW}
\{(\be,C)\in s(x)\mathfrak{p} s(x_{i+1})\mid \be\neq x_1\ldots x_i\}=\emptyset   
\end{equation}
for every $i\in\N$. Put $A=D_{(x_1,r(x_1))}$, and suppose that $\pi_U\in \llbracket \Gs\rrbracket$ is such that $x\in \supp(\pi_U)\subseteq D_{(x_1,r(x_1))}$. Then there is a basic open neighborhood $\mathcal{Z}(\al,\be,B,\emptyset)\subseteq U$ (Proposition \ref{prop:characbisecandfullgrp}) with $(\al,B),(\be,B)\in\mathfrak{p}$, $\al\neq\be$ and $x\in D_{(\be,B)}\subseteq\supp(\pi_U)\subseteq D_{(x_1,r(x_1))}$. Hence $\be=x_1\cdots x_m$ for some $m\in \N$. Since $D_{(\al,B)}\subseteq\supp(\pi_U)\subseteq D_{(x_1,r(x_1))}$, it follows that $s(x)=s(\al)=s(\be)$ and $s(x_{m+1})\in B$. Thus $(\al,B)\in \{(\be,C)\in s(x)\mathfrak{p} s(x_{i+1})\mid \be\neq x_1\ldots x_i\} $, which contradicts  (\ref{eqn:NotConditionW}), because $\al\neq \be=x_1\cdots x_m$. Thus (F1) is not satisfied. 

Finally, we show the necessity of Condition~($\infty$) for (F1) to be satisfied. Assume that Condition~($\infty$) does not hold. Then there is a minimal infinite emitter $B\in\GG^0$ such that the set $F:=\{e\in \varepsilon(B):r(e) \mathfrak{p} B\neq \emptyset\}$  is finite. Put $x=(B,B)$ and let $A=D_{(B,B)F}$. We claim that (F1) fails for this pair. To see this, suppose that $\pi_U\in \llbracket \Gs\rrbracket$ is such that $x\in \supp(\pi_U)\subseteq D_{(B,B)F}$. Then there is a basic open neighborhood $\mathcal{Z}(\al,\be,C,\emptyset)\subseteq U$ (Proposition \ref{prop:characbisecandfullgrp}) with $(\al,C),(\be,C)\in\mathfrak{p}$, $\al\neq\be$, $(B,B)\in D_{(\al,C)}\subseteq D_{(B,B)F}$  and $(B,B)\in D_{(\be,C)}\subseteq D_{(B,B)F}$. Since $(B,B)\in D_{(\al,C)}\subseteq D_{(B,B)F}$  and $(B,B)\in D_{(\be,C)}\subseteq D_{(B,B)F}$, it follows that $B\subseteq C$ and that $\al=\be=B$, which contradicts that $\al\neq\be$ and completes the proof.
\end{proof}



\begin{lema}\label{lem:InfOrbit}
Let $\GG$  be an ultragraph with no sinks that satisfy Conditions (RFUM), (K), (W) and ($\infty$).  Then $\mathrm{Orb}_{\Gs}(x)$  is infinite for each $x\in X$. 
\end{lema}
\begin{proof}
The proof is exactly the same as \cite[Lemma 10.9]{MR3950815}.
\end{proof}

We can now state the first main result of this section. 

\begin{teorema}\label{thm:MainIsoUGs}
Let $\GG$ and $\mathcal{F}$ be ultragraphs with no sinks that satisfy Conditions (RFUM), (K), (W) and ($\infty$). Let $\Gamma$ be a subgroup of $\llbracket \Gs\rrbracket$ containing the commutator subgroup $D(\llbracket \Gs\rrbracket)$ and let $\La$ be a subgroup of $\llbracket \mathcal{F}_\sigma\rrbracket$ containing the commutator subgroup $D(\llbracket \mathcal{F}_\sigma\rrbracket)$. Then the following are equivalent:
\begin{enumerate}
    \item $\Gs\cong \mathcal{F}_\sigma$ as topological groupoids. 
    \item $\llbracket \Gs\rrbracket \cong \llbracket \mathcal{F}_\sigma\rrbracket$  as abstract groups.
    \item $D(\llbracket \Gs\rrbracket) \cong D(\llbracket \mathcal{F}_\sigma\rrbracket)$  as abstract groups.
\end{enumerate}
\end{teorema}
\begin{proof}
Let $X$ denote the edge shift space of $\GG$ and let $Y$ be the edge shift space of $\mathcal{F}$. 
It follows from Theorem \ref{thm:main1KWInf} that $(\Gamma, X)\in K^F$ and $(\Lambda, Y)\in K^F$. By \cite[Theorem 6.6]{MR3950815} the class $K^F$ is a faithful class of space-group pairs. The result now follows directly from  \cite[Proposition 6.2]{MR3950815}, Lemma \ref{lem:InfOrbit},  \cite[Lemma 4.9]{MR3950815} and \cite[Proposition 4.10]{MR3950815}. 
\end{proof}

\begin{exemplo}
We give an example of an ultragraph $\GG $  with property (RFUM) such that the ultragraph $C^*$-algebra associated with $\GG$ is not isomorphic to any graph $C^*$-algebra \cite[Remark 4.4]{MR2020023}, and which satisfies Conditions (K), (W) and ($\infty$). 
    Let $\GG$ be the ultragraph associated with the matrix $A$ given by  $$A(i,j)=\begin{cases} 
    1, \text {if $i=j$, or $i=j+2$, or $i\in\{1,2\}$ and $j\geq 3$},\\
    0 \text{ otherwise.}
    \end{cases}
    $$
    Then $\GG$ is given by a countable number of vertices, say $\{v_i\}$, and a countable number of edges, say $\{e_i\}$, such that $s(e_i)=v_i$ for all $i$, $r(e_1)= \{v_i:i\neq 2\}$, $r(e_2)= \{v_i:i\neq 1\}$, and, for $n\geq 3$,
    $r(e_n)= \{v_{n-2}, v_n\}$. It is straigtforward to check that $\GG$ satisfies Condition~(K). The only minimal infinite emitter is $r(e_1)\cap r(e_2)$ and Condition~$(\infty)$ follows. Finally notice that $\GG$ has no wandering path, and thus satisfies Condition~(W).
\end{exemplo}

Our next goal is to prove another isomorphism theorem with slightly weaker conditions than in Theorem \ref{thm:MainIsoUGs}. However, as a result of this weakening we lose the isomorphism of subgroups (as is already evident in the case of graphs, \cite[Theroem 10.11]{MR3950815}).

\begin{definicao}
    $\GG$ satisfies Condition (T) if for every vertex  $v\in G^0$, there exists a vertex $w\in G^0$ such that
    \begin{equation}\label{conditionT}
        |\{\be\in\GG^*\mid\text{there exists }C\text{ such that }(\be,C)\in \{v\}\mathfrak{p} \{w\}|\geq 2.
    \end{equation}
\end{definicao}

\begin{lema}
Let $\GG$  be an ultragraph with no sinks that satisfies Condition~(RFUM). The groupoid $\Gs$ is non-wandering if and only if $\GG$ satisfies Conditions (L) and (T).
\end{lema}

\begin{proof}
The proof follows the same line of \cite[Proposition 10.7(i)]{MR3950815}. To show that Conditions (L) and (T) imply that $\Gs$ is non-wandering, we merely have to consider our cylinder set $D_{(\mu,r(\mu))}$, instead of $Z(\mu)$ in \cite[Proposition 10.7(i)]{MR3950815}.

If condition (L) does not hold, then $X$ has isolated points, by Proposition \ref{prop:isoltedpoints}, and thus $\Gs$ is wandering.

Now, suppose Condition (T) is not true and let $v$ be a vertex where the condition fails; that is, for each vertex $w$, either there is no path connecting $v$ to $w$ or there exists only one path connecting $v$ to $w$. For $(\alpha,A)\in X$ such $s(\alpha)=v$, since $\alpha$ connects $v$ to any element of $A$, $\alpha$ is the only path with this property. So if $(\beta,B)\in D_{(\{v\},\{v\})}\cap \text{Orb}_{\Gs}(\alpha,A)$, then $s(\beta)=v$ and $B=A$, so that $\beta$ connects $v$ to any element of $A$, and therefore $\beta=\alpha$. And for an infinite path $x\in X$, the proof that $(\beta,B)\in D_{(\{v\},\{v\})}=\{x\}$ is the same as in \cite[Proposition 10.7]{MR3950815}. We again conclude that $\Gs$ is wandering.

\end{proof}

\begin{definicao}\label{def:degenerate}
    Let $\GG$ be an ultragraph with no sinks. We say that a minimal infinite emitter $A\in\GG^0$ is degenerate if it satisfies one of the two following conditions:
    \begin{itemize}
        \item[IE1.] $v$ is a source for every $v\in A$.
        \item[IE2.] There exists a unique $v\in A$ such that $v$ is not a source, and for this $v$ we have that $\GG^1v=\{e\}$ and $s(e)$ is a source.
    \end{itemize}
    And we say that a vertex $v\in G^0$ is degenerate if one of the following conditions is satsified:
    \begin{itemize}
        \item[V1.] $\GG^1v=\{e\}=v\GG^1v$.
        \item[V2.] $\GG^1v=\{e,f\}$, where $e\in v\GG^1$ and $s(f)$ is a source.
        \item[V3.] There exists a vertex $w$ different from $v$ such that there are edges $e,f$ with $\GG^1v=\{f\}=w\GG^1v$ and $\GG^1w=\{e\}=v\GG^1w$.
    \end{itemize}
    
    Finally, $\GG$ satisfies Condition (ND) if there are no degenerate infinite emitters and no degenerate vertices.
\end{definicao}

\begin{lema}\label{OrbCard}
    Let $\GG$  be an ultragraph with no sinks that satisfies Condition~(RFUM). Then $|\mathrm{Orb}_{\Gs}(x)|\geq 3$ for all $x\in X$ if and only if $\GG$ satisfies Condition (ND).
\end{lema}

\begin{proof}
    Recalling the definition of orbit given in Subsection \ref{subsec:groupoid.full.group}, we see that $y\in \mathrm{Orb}_{\Gs}(x)$ if and only if $x$ and $y$ have the same tail. 

    An element $x$ such that $|\mathrm{Orb}_{\Gs}(x)|=1$ must be of the form $x=e^{\infty}$, where $s(e)$ is a degenerate vertex satisfying V1 from Definition \ref{def:degenerate}, or it is of the form $x=(A,A)$, where $A$ is a minimal infinite emitter satisfying IE1 from Definition \ref{def:degenerate}.
    
    Now, if $x$ is such that $|\mathrm{Orb}_{\Gs}(x)|=2$, then there are three possibilities for $\mathrm{Orb}_{\Gs}(x)$, namely, $\{(A,A),(e,A)\}$, $\{e^
    {\infty},fe^{\infty}\}$ or $\{(ef)^{\infty},(fe)^{\infty}\}$, from where we get Conditions IE2, V2 and V3 of Definition \ref{def:degenerate}, respectively.
\end{proof}

\begin{teorema}\label{main2}
    Let $\GG$ and $\mathcal{F}$ be ultragraphs with no sinks that satisfy Conditions (RFUM), (L), (T) and (ND). Then the following are equivalent:
    \begin{enumerate}
    \item $\Gs\cong \mathcal{F}_\sigma$ as topological groupoids. 
    \item $\llbracket \Gs\rrbracket \cong \llbracket \mathcal{F}_\sigma\rrbracket$  as abstract groups.
\end{enumerate}
\end{teorema}

\begin{proof}
    Since the groupoids $\Gs$ and $\mathcal{F}_\sigma$ are ample, the result follows immediately from Lemma \ref{OrbCard} and \cite[Theorem 7.10]{MR3950815}.
\end{proof}



\section{Acknowledgements}

The majority of this paper was completed while the third author worked at Universidade Federal de Santa Catarina under the guidance of the first two authors. He thanks them for their guidance and warm hospitality.

\bibliographystyle{abbrv}
\bibliography{references}

\vspace{1.5pc}

Gilles Gon\c{c}alves de Castro, Departamento de Matem\'atica, Universidade Federal de Santa Catarina, Florian\'opolis, 88040-900, Brazil.

Email: gilles.castro@ufsc.br	

\vspace{0.5pc}

Daniel Gon\c{c}alves, Departamento de Matem\'atica, Universidade Federal de Santa Catarina, Florian\'opolis, 88040-900, Brazil.

Email: daemig@gmail.com

\vspace{0.5pc}

Daniel W van Wyk, Department of Mathematics, Dartmouth College, Hanover, NH 03755-3551, USA.

Email: daniel.w.van.wyk@dartmouth.edu

\end{document}